\documentclass[12pt,reqno]{amsart}

\usepackage{amsmath,amssymb,amsthm,amsfonts}
\usepackage{bm}
\usepackage{natbib}
\usepackage{graphicx}
\usepackage{algorithm,algorithmic}

\usepackage{fullpage,amsaddr}
\usepackage{mathrsfs}

\newtheorem{theorem}{Theorem}
\newtheorem{corollary}[theorem]{Corollary}
\newtheorem{lemma}[theorem]{Lemma}
\newtheorem{proposition}[theorem]{Proposition}
\newtheorem{assumption}{Assumption}
\theoremstyle{definition}
\newtheorem{definition}{Definition}
\newtheorem{example}{Example}

\theoremstyle{remark}
\newtheorem{remark}{Remark}
\theoremstyle{remark}

\newcommand{\R}{\mathbb{R}}
\newcommand{\N}{\mathbb{N}}

\newcommand{\Ep}{\mathbb{E}}

\renewcommand{\hat}{\widehat}
\renewcommand{\tilde}{\widetilde}

\newcommand{\argmin}{\operatornamewithlimits{argmin}}

\newcommand{\mone}{\textbf{1}}

\newcommand{\C}{\mathbb{C}}

\newcommand{\trace}{\mathrm{tr}}
\newcommand{\vectorize}{\mathrm{vec}}
\newcommand{\diag}{\mathrm{diag}}

\usepackage{color}

\title{Benign Overfitting in Time Series Linear Model with Over-Parameterization}

\author{Shogo Nakakita$^\dagger$ \and Masaaki Imaizumi$^{\dagger \ddagger}$}

\address{$^\dagger$The University of Tokyo, $^\ddagger$RIKEN Center for AIP}

\begin{document}

\maketitle

\footnote{
E-mail: imaizumi@g.ecc.u-tokyo.ac.jp.

$^\dagger$ The University of Tokyo: 3-8-1 Komaba, Meguro, Tokyo, 153-0041 JAPAN.

$\ddagger$ RIKEN Center for AIP: 1-4-1 Nihonbashi, Chuo, Tokyo, 103-0027 JAPAN.
}

\begin{abstract}
The success of large-scale models in recent years has increased the importance of statistical models with numerous parameters. Several studies have analyzed over-parameterized linear models with high-dimensional data, which may not be sparse; however, existing results rely on the assumption of sample independence. In this study, we analyze a linear regression model with dependent time-series data in an over-parameterized setting. We consider an estimator using interpolation and develop a theory for the excess risk of the estimator. Then, we derive non-asymptotic risk bounds for the estimator for cases with dependent data. This analysis reveals that the coherence of the temporal covariance plays a key role; the risk bound is influenced by the product of temporal covariance matrices at different time steps. Moreover, we show the convergence rate of the risk bound and demonstrate that it is also influenced by the coherence of the temporal covariance. Finally, we provide several examples of specific dependent processes applicable to our setting.
\end{abstract}

\section{Introduction}

In this study, we analyze the high-dimensional time-series linear regression problem with dependent covariates.
Suppose there are $n$ covariates, $x_1,...,x_n \in \R^p$ with dimension $p \in \N$, and $n$ responses, $y_1,...,y_n \in \R$, observed from the following linear model with an unknown true parameter $\beta^* \in \R^p$:
\begin{align}
    y_t = \langle \beta^*, x_t \rangle + \varepsilon_t, ~ t=1,...,n, \label{def:model}
\end{align}
where $\langle \cdot, \cdot \rangle$ denotes the inner product in $\R^p$.
Here, the covariates $\{x_t: t = 1,...,n\}$ and the independent noise $\{\varepsilon_t: t = 1,...,n\}$ are stationary, centered Gaussian processes that are independent of each other.
It should be noted that we allow for dependence within $\{x_t: t = 1,...,n\}$ and $\{\varepsilon_t: t = 1,...,n\}$.
We focus on the over-parameterized case, where the dimension $p$ is significantly larger than $n$ (and possibly infinite).
In this study, we derive sufficient conditions under which the risk of an estimator for the model converges to zero as $n \to \infty$, indicating benign overfitting, under several settings of dependence within $\{x_t: t = 1,...,n\}$.

Statistical methods for high-dimensional and large-scale data analysis have garnered significant attention over the past decades. A prominent approach leverages sparsity through $\ell_1$ norm regularization and its variants \citep{candes2007dantzig,van2008high,buhlmann2011statistics,hastie2019statistical}, which is effective when the target signal contains many zero elements.
Another line of research has explored the asymptotic risk behavior of estimators in high-dimensional settings \citep{dobriban2018high,belkin2019reconciling,hastie2019surprises,bartlett2020benign,tsigler2023benign,nakakita2024dimension} and has revealed the limiting risk behavior when the number of data points $n$ and parameters $p$ tend to infinity while their ratio $p/n$ converges to a positive constant.
Interpolators, which are estimators that perfectly fit the observed data in the $p \gg n$ regime, have been shown to have risk or variance that converges to zero \citep{liang2020just,hastie2019surprises,ba2019generalization,tsuda2024benign}.
Notably, these studies did not exploit the sparsity of data or signals; however, they remain compatible with recent large-scale data analysis.

In the context of high-dimensional time-series data, several challenges remain unresolved.
The first challenge is that no consistent estimator has been established unless sparsity is imposed.
Although several consistent estimators with sparsity exist \citep{wang2007regression,Basu2015Regularized,han2015direct,Wong2020Lasso}, demonstrating the consistency of estimators without sparsity is a non-trivial task, as discussed in \citet{daskalakis2019regression,kandiros2021statistical}.
This difficulty arises because many probabilistic tools used in high-dimensional statistics rely on the assumption of variable independence.
The second challenge is identifying a statistic that influences the estimation risk in this setting.
In the case of independent data, \cite{bartlett2020benign} and \cite{tsigler2023benign} showed that the effective rank explains the risk.
For low-dimensional or sparse high-dimensional time-series data, it has been shown that temporal covariance matrices characterize the risk of estimators \citep{Basu2015Regularized,Wong2020Lasso}.
However, in high-dimensional time-series settings without sparsity, these statistics are not valid measures.

In this study, we investigate the risk of an interpolation estimator in stochastic linear regression \eqref{def:model} with high-dimensional dependent processes to tackle these challenges.
Specifically, we introduce an assumption on the coherence of temporal covariance matrices and derive non-asymptotic bounds for the prediction risk of the interpolation estimator, motivated by recent advances in neural networks that achieve perfect data fitting.
We further demonstrate the convergence rate of the derived bound as $n \to \infty$ under tail assumptions on the eigenvalues of the covariance matrix of the data.
These results hold when the data dimension is sufficiently large, under a condition on the covariance matrix.
This setup includes the following cases: a stationary Gaussian process $\{x_t\}_{t=1}^n$, which can be (i) a separable vector autoregressive moving-average (ARMA) process; (ii) a non-separable vector ARMA process; or (iii) a time-varying regression model under temporal dependence.

Our results address the challenges of high-dimensional time-series analysis, as summarized below:
(i) We develop a consistency theory for the estimator in a high-dimensional time-series regression model \eqref{def:model} without assuming sparsity of the target parameter.
Specifically, we clarify that the interpolation estimator is consistent under a spectral condition on the covariance matrix in the time direction, in addition to the spectral condition on the covariance matrix in the spatial direction.
(ii) We identify that \textit{coherent temporal covariance}, defined as the product of temporal covariance matrices with a reference matrix, characterizes the risk in high-dimensional time-series estimation.
Specifically, we denote the covariance matrix of the $k$-th component of the covariate as $\Xi_{k,n}$ and an arbitrary reference matrix as $\Xi_{0,n}$, and we show that ensuring the eigenvalues of $\Xi_{k,n}^{-1}\Xi_{0,n}$ remain finite and bounded away from zero is crucial for risk control.
Unlike the spatial-directional condition identified by \cite{bartlett2020benign}, the temporal-directional condition we derive is expressed in terms of the eigenvalues of a matrix product.
Furthermore, this property has not been addressed in studies on high-dimensional time-series models with sparsity, e.g., \citet{Basu2015Regularized}, and it represents a unique characteristic of our over-parameterized setup.

On the technical side, our core results rely on two key contributions.
First, we observe that the interpolation estimator in linear regression can implicitly decorrelate observations.
This result alone provides insight into benign overfitting when the covariate process has a separable structure.
Second, we analyze the spectral structure of temporally dependent random matrices.
This represents a sophisticated extension of the analyses by \citet{bartlett2020benign} and \citet{tsigler2023benign} to time series models.
Specifically, we control the deviation of the spectral structure from the i.i.d.~case and show that it asymptotically coincides up to constant factors.

\subsection{Related studies} 
In this paper, we discuss the following two types of related studies.

\textit{High-dimensional time-series}:
Time-series analysis of high-dimensional data has been conducted for stochastic regression and vector auto-regression (VAR) estimation problems using various sparsity-induced regularizations.
These studies applied regularization under the assumption of sparsity.
\citet{wang2007regression} proposed a flexible parameter tuning method for the lasso that can be applied to VAR.
\citet{alquier2011sparsity} studied a general scheme of $\ell_1$-norm regularization for dependent noise with a wide class of loss functions.
\citet{song2011large} focused on large-scale VAR and its sparse estimation.
\citet{Basu2015Regularized} developed a spectrum-based characterization of dependent data and employed a restricted eigenvalue condition for $\ell_1$-norm regularization of dependent data.
\citet{kock2015oracle} studied the lasso for time-series data and established an oracle inequality and model selection consistency.
\citet{han2015direct} employed the Danzig selector for high-dimensional VAR and analyzed its efficiency based on the spectral norm of a transition matrix.
\citet{wu2016performance} developed a widely applicable risk analysis method for time-series lasso under various tail probability distributions.
\citet{guo2016high} focused on applying the lasso to VAR with banded transition matrices.
\citet{davis2016sparse} proposed a two-step estimation algorithm for VAR with sparse coefficients.
\citet{medeiros2016L1} studied time-series lasso with non-Gaussian and heteroscedastic covariates.
\citet{masini2019regularized} derived an oracle inequality for sparse VAR with fat probability tails and various dependent settings.
\citet{Wong2020Lasso} studied time-series lasso with general tail probability and dependence by relaxing a restriction in \citet{Basu2015Regularized}.
These studies imposed sparsity on the data; thus, they are applicable to non-sparse high-dimensional data in our setting.
\citet{bunea2022interpolating} studied a factor model with over-parameterization within the framework of benign overfitting.

\textit{Large-scale models with dependent data}:
In contrast to studies focusing on sparsity, few studies have analyzed dependent data in the context of large-scale models. Learning theory has advanced predictive risk analysis of dependent data through uniform convergence and complexity measures. Unlike sparsity-based approaches, these methods can be applied relatively easily to large-scale models.
\citet{yu1994rates} demonstrated uniform convergence of empirical processes under mixing conditions and derived its convergence rate. \citet{mohri2008rademacher} introduced Rademacher complexity to analyze predictive risks in stationary mixing processes. \citet{berti2009rate} established uniform convergence of empirical processes under exchangeability. \citet{mohri2010stability} proved data-dependent Rademacher complexity bounds for various types of mixing processes. \citet{agarwal2012generalization} derived upper bounds on generalization errors for general loss functions and mixing processes. \citet{kuznetsov2015learning} established an upper bound on predictive errors for stochastic processes without assuming stationarity or mixing. \citet{dagan2019learning} derived a generalization error bound for dependent data that satisfied the Dobrushin condition. \citet{daskalakis2019regression} and \citet{kandiros2021statistical} applied the Ising model, analyzed several linear models with dependent data, and derived novel estimation risk bounds.
These studies yielded general results, but they did not consider scenarios where the number of parameters tends to infinity, such as in over-parameterization.

\subsection{Organization}

Section \ref{sec:setting} presents the stochastic regression problem and defines the estimators with interpolation.
Section \ref{sec:result} develops a non-asymptotic risk bound for the estimator and introduces the assumptions.
Section \ref{sec:benign} derives the convergence rate of the risk under an additional assumption on the covariance matrix of the data.
Section \ref{sec:derivation} provides an overview of the proofs of the main results.
Section \ref{sec:example} presents several examples of Gaussian processes as covariates.
Section \ref{sec:conclusion} states the conclusions of this study.
The appendix contains the main proofs. 

\subsection{Notation}
For a vector $b \in \R^d$, $b^{(i)}$ denotes the $i$-th element of $b$ for $i=1,...,d$, and $\|b\|^2 := \sum_{i=1}^d (b^{(i)})^2$ denote the Euclidean norm.
For a square matrix $A \in \R^{d}\otimes \R^d$, $A^{(i_1,i_2)} \in \R$ denotes an element on the $i_1$-th row and $i_2$-th column of $A$ for $i_1,i_2 = 1,...,d$.
$\|A\| := \sup_{z \in \R^d: \|z\|=1} \|T z\|$ is an operator norm of $A$, and $\mu_i(A)$ is the $i$-th largest eigenvalue of $A$, and $\trace(A) = \sum_{i} \mu_i(A)$ denotes a trace of $A$.
For positive sequences $\{a_n\}_n$ and $\{b_n\}_n$, $a_n \lesssim b_n$ and $a_n = O(b_n)$ indicate that there exists $C>0$ such that $a_n \leq C b_n$ for any $n \geq \Bar{n}$ with some $\Bar{n} \in \N$.
$a_n \prec b_n$ and $a_n = o(b_n)$ denote that for any $C>0$, $a_n \leq C b_n$ holds true for any $n \geq \Bar{n}$ with some $\Bar{n} \in \N$.
$a_n \asymp b_n$ denotes that both $a_n \lesssim b_n$ and $b_n \lesssim a_n$ hold true.
For $a \in \R$, $\log^a n$ denotes $(\log n)^a$.
For an event $E$, $\mone\{E\}$ is an indicator function such that $\mone\{E\} = 1$ if $E$ is true and $\mone\{E\} = 0$ otherwise.
For arbitrary random variables $r_{1},\ldots,r_{i}$, let $\Ep_{r_{1},\ldots,r_{i}}$ denote the conditional expectation given all random variables other than $r_{1},\ldots,r_{i}$.
$\N_{0}$ indicates the union of $\N$ and $\left\{0\right\}$.
For a positive semi-definite matrix with a form $\Sigma=U\Lambda U^{\top}$ with a diagonal matrix $\Lambda$ and an orthogonal matrix $U$, we define $\Sigma_{0:k}:=[e_{1},\ldots,e_{k}]\mathrm{diag}\{\lambda_{1},\ldots,\lambda_{k}\}[e_{1},\ldots,e_{k}]^{\top}$ and $\Sigma_{k:\infty}:=[e_{k+1},\ldots,e_{p}]\mathrm{diag}\{\lambda_{k+1},\ldots,\lambda_{p}\}[e_{k+1},\ldots,e_{p}]^{\top}$
where $e_{i}$ are the column vectors of $U=[e_{1},\ldots,e_{p}]$. 

\section{Setting and assumption} \label{sec:setting}

\subsection{Gaussian process with spatio-temporal covariance} \label{sec:setting_process}

We first define a $p$-dimensional covariate process $\{x_{t}\}_{t=1}^n$ and the corresponding design matrix $X=[x_{1},\ldots,x_{n}]^{\top}\in\R^{n}\otimes\R^{p}$. We construct the covariate process by multiplying temporal and spatial covariance matrices with a random matrix whose entries are i.i.d. This approach generalizes separable models in spatial and spatio-temporal statistics \citep[for details, see][]{Cre93,KJ99,CGS21} and encompasses studies on random matrices with temporal dependence \citep{couillet2015rmt,tian2022joint,tian2022ratio}.

Consider positive definite matrices $\Xi_{i,n}\in\R^{n}\otimes\R^{n}$, $i=1,\ldots,p$ and $\Sigma\in\R^{p}\otimes\R^{p}$, that represent the temporal and the spatial covariance, respectively.
We also introduce
$Z=[z_{1},\ldots,z_{p}]$, $z_{i}\sim N(\mathbf{0},I_{n})$ as a matrix with standard Gaussian i.i.d. entries.
Let $\lambda_1 \geq \lambda_2 \geq \cdots \geq \lambda_p > 0$ be the eigenvalues of $\Sigma$, $\Lambda=\mathrm{diag}\{\lambda_{1},\ldots,\lambda_{p}\}$ be a diagonal matrix, and $U\in \R^{p}\otimes \R^{p}$ be an orthogonal matrix such that $\Sigma=U\Lambda U^{\top}$.
Then, we set the design matrix $X=[x_{1},\ldots,x_{n}]^{\top}\in\R^{n}\otimes\R^{p}$ for the covariates as
\begin{align}
 X=\left[\Xi_{1,n}^{1/2}z_{1},\ldots,\Xi_{p,n}^{1/2}z_{p}\right]\Lambda^{1/2}U^{\top}. \label{def:proceess}
\end{align}
If there exists $\Xi_n \in \R^n \otimes \R^n$ with $\Xi_{i,n} = \Xi_n$ for all $i=1,...,n$, the definition \eqref{def:proceess} is simplified to
\begin{align}
    X = \Xi_n^{1/2} Z \Lambda^{1/2} U^\top, \label{def:process_homo}
\end{align}
which we call a separable model.
The covariate process $\{x_{t}\}_{t=1}^n$ has the  spectral decomposition, 
that is, $\{x_{t}^{\top}u_{i}\}_{t}$ and $\{x_{t}^{\top}u_{j}\}_{t}$ with the column vectors $\{u_{i}\}$ of $U$ and $i\neq j$ are independent of each other.
We set this independence to yield the concentration inequalities for some random matrices, which play important roles in our analysis (see Appendix \ref{sec:proof_hetero}).
This property appears in some physical models with stochastic partial differentiable equations; see \citet{lototsky2009statistical,lototsky2017stochastic}.

The process has simple (possibly lagged) autocovariance matrices as follows: for all $t_{1},t_{2} \in \{1,2,...,n\}$, the autocovariance matrix is written as
\begin{align}
    \Ep\left[x_{t_{1}}x_{t_{2}}^{\top}\right]=U\mathrm{diag}\left\{\Xi_{1,n}^{(t_{1},t_{2})}\lambda_{1},\ldots,\Xi_{p,n}^{(t_{1},t_{2})}\lambda_{p}\right\}U^{\top}. \label{def:covariance}
\end{align}
If all $\Xi_{i,n}$ are Toeplitz matrices, i.e. their $(t_1,t_2)$-th elements dependent only on $|t_{1}-t_{2}|$, the covariance \eqref{def:covariance} is further simplified.
Note that the autocovariance matrix of $x_{t}$, i.e., $\Ep [x_{t}x_{t}^{\top}]$, is not necessarily equal to the spatial covariance matrix $\Sigma$.  
This fact does not matter in our study. 

We present an auto-regressive moving-average (ARMA) process as simple examples of processes with these structures.
Let $\{w_{t}\}_{t \in \N}$ be a sequence of independent $p$-dimensional Gaussian vectors whose mean is zero and covariance matrix is $Q$ with $\mathrm{rank}{(Q)} > n$.

\begin{example}[vector ARMA process] \label{example:vector_arma}
We present a vector ARMA process with the structure \eqref{def:proceess}:
\begin{align}
    x_{t}=\sum_{j=1}^{\ell_1}\rho_{j}x_{t-j} + w_{t} + \sum_{j=1}^{\ell_2}\varphi_{j}w_{t-j},\label{def:arma_hetero}
\end{align}
where $\rho_j,\varphi_j \in \R^p \otimes \R^p$ are coefficient matrices with spectral decompositions with an orthogonal basis $\left\{e_{k}\right\}$ such that $\rho_{j}=\sum_{k=1}^{\ell_{1}}\rho_{j,k}e_{k}e_{k}^{\top}$, and $\varphi_{j}=\sum_{k=1}^{\ell_{2}}\varphi_{j,k}e_{k}e_{k}^{\top}$.
We also assume that $Q=\sum_{k=1}^{p}q_{k}e_{k}e_{k}^{\top}$ and represent $e_{k}^{\top}w_{t}=\sqrt{q_{k}}z_{t,k}$ with i.i.d.\ $z_{t,k}\sim N(0,1)$.
Here, we define $\varpi_{1,k}\left(z\right)=1-\sum_{j=1}^{\ell_1}\rho_{j,k}z^{j}$ and $\varpi_{2,k}\left(z\right)=1+\sum_{j=1}^{\ell_2}\varphi_{j,k}z^{j}$ as characteristic polynomials of the AR and MA parts, respectively, and assume that $\varpi_{1,k}\left(z\right)\neq0$ for all $z$ with $|z|\le 1$ and they do not have roots on the unit circle.
There exist causal MA representations $e_{k}^{\top}x_{t}=\sum_{j=0}^{\infty}\phi_{j,k}e_{k}^{\top}w_{t-j}$ with $\{\phi_{j,k}\in\R\}$ such that $\varpi_{2,k}(z)/\varpi_{1,k}(z)=\sum_{j=0}^{\infty}\phi_{j,k}z^{j}$ \citep[Theorem 3.1.1]{BD1991Time}.
It holds that $\Xi_{k,n}^{(t,t+h)}=\sum_{j=0}^{\infty}\phi_{j,k}\phi_{j+|h|,k}/\sum_{j=0}^{\infty}\phi_{j,k}^{2}$ and $\Sigma=\sum_{k=1}^{p}q_{k}(\sum_{j=0}^{\infty}\phi_{j,k}^{2})e_{k}e_{k}^{\top}$.
\end{example}

\begin{remark}
    Note that the representation of $\Xi_{k,n}$ and $\Sigma$ are not unique.
    We see that $\Xi_{k,n}^{\rm new}:=c\Xi_{k,n}$ and $\Sigma^{\rm new}:= c^{-1}\Sigma$ with a scaling factor $c>0$ give the same covariance structure for $x_{t}$ as $\Xi_{k,n}$ and $\Sigma$.
    One of the reasons to introduce the MA representation with $\{\phi_{j,k}\}$ is to let these representations be unique.
\end{remark}

\subsection{Stochastic regression problem and interpolation estimator}

We consider the stochastic regression problem with the covariate process $\{x_t\}_{t=1}^n$ defined above.
Suppose that we have $n$ observations, $(x_1,y_1),...,(x_n,y_n) \in \R^p \times \R$ with dimension $p$, which follow the stochastic linear regression model \eqref{def:model} for $t=1,...,n$ with an unknown true parameter, $\beta^* \in \R^p$.
$\{\varepsilon_t : t =1,...,n\}$ is a Gaussian process such that $\mathcal{E}:=(\varepsilon_{1},\ldots,\varepsilon_{n})^{\top}\sim N\left(\mathbf{0},\Upsilon_{n}\right)$, where $\Upsilon_{n}\in\R^{n} \otimes \R^{n}$ is a positive definite matrix denoting the autocovariance of $\varepsilon_{t}$. 
Note that $x_{t}$ and $\varepsilon_{t}$ are independent of each other.
We define $Y = (y_1,...,y_n)^\top \in \R^n$ as a design vector.

We consider an interpolation estimator with the minimum norm:
\begin{align}
    \hat{\beta} \in \argmin_{\beta \in \R^p} \|\beta\|^2 \mbox{~~s.t.~~} \textstyle\sum_{t=1}^n (y_t - \langle \beta, x_t \rangle)^2 = 0, \label{def:interpolator}
\end{align}
with norm $\|\cdot\|$ for $\R^p$.
A solution satisfying this constraint in \eqref{def:interpolator} is guaranteed to exist if $\{x_t\}_{t=1}^n$ spans an $n$-dimensional linear space, i.e., if $\R^p$ has a larger dimensionality than $n$ (for example,  $\R^p = \R^p$ and $p \geq  n$), then multiple solutions satisfy the linear equation $ Y= \langle X, \beta \rangle$.
The interpolation estimator \eqref{def:interpolator} can be rewritten as
\begin{align}
    \hat{\beta} = (X^\top X)^\dagger X^\top Y = X^\top (XX^\top)^{-1} Y, \label{eq:betahat} 
\end{align}
where $^\dagger$ denotes the pseudo-inverse of matrices.

We study an excess prediction risk of the interpolation estimator $\hat{\beta}$ as 
\begin{align}
    R(\hat{\beta}) := \Ep^* [(y^* - \langle \hat{\beta}, x^* \rangle )^2 - (y^* - \langle {\beta}^*, x^* \rangle )^2], \label{def:risk}
\end{align}
where $(x^\ast, \varepsilon^\ast)$ is an independent random element such that $x^\ast \sim N(\mathbf{0},\Sigma)$, $\varepsilon^\ast\sim N(0,\sigma_\varepsilon^2)$ with arbitrary $\sigma_\varepsilon>0$, and $y^\ast=(x^\ast)^\top \beta^\ast+\varepsilon^\ast$ and  expectation  $\Ep^*[\cdot]$ is taken with respect to $(x^*, y^*)$.
The setting of $(x^*, y^*)$ can be justified as an approximation of $(x_{T},y_{T})$ with large $T$ when $x_{t}$ is mixing, or $x_{t}$ and $x_{t+h}$ with large $h$ are asymptotically independent.

\subsection{Effective ranks and numbers}

In preparation for our analysis, we define several notions related to matrices.
The first is the \textit{effective rank} of matrices, which measures the complexity of a matrix using its eigenvalues.
\begin{definition}[effective rank] \label{def:effective_rank}
For a matrix $T \in \R^{p} \otimes \R^{p}$ and $k \in  \N$, we define two types of effective ranks of $T$:
\begin{align*}
    r_k(T) = \frac{\sum_{i > k} \mu_i(T)}{ \mu_{k+1}(T)}, \mbox{~and~} R_k(T) = \frac{( \Sigma_{i > k} \mu_i(T))^2}{(\sum_{i > k} \mu_i(T)^2)}.
\end{align*}
Furthermore, we define an effective number of bases with constants $b > 0$ and $n \in \N$:
\begin{align*}
    k^*(b) = \min \{k \ge 0: r_k(\Sigma) \geq bn \}.
\end{align*}
If the set is empty, we set $k^*(b) = \infty$.
\end{definition}
This notion has been used in studies related to the matrix concentration inequality \citep{koltchinskii2017concentration}, and its application to high-dimensional linear regression \citep{bartlett2020benign}. This measure is based on the degree of decay of the eigenvalues, rather than the number of nonzero eigenvalues.

\section{Non-asymptotic risk analysis} \label{sec:result}

We present a non-asymptotic analysis of the excess risk of the interpolation estimator. Subsequently, we provide a sufficient condition on the covariance matrices of the covariates for the excess risk to converge to zero.

\subsection{Assumption: coherent temporal covariance}

We provide an assumption for our analysis of the excess risk. Importantly, we require an assumption regarding the coherence of the temporal covariance $\Xi_{0,n}$, as defined in \eqref{def:proceess}, as stated below.
\begin{assumption}[coherent temporal covariance]\label{asmp:coherence}
    There exist a positive number $\epsilon\in(0,1]$ and a family of $n\times n$ positive definite matrices $\Xi_{0,n} \in \R^n \otimes \R^n$ (reference matrices) such that for all $n\in\N$ and $i=1,\ldots,p$, we have
    \begin{align*}
        \epsilon\le \inf_{i,n}\mu_{n}(\Xi_{0,n}^{-1}\Xi_{i,n})\le \sup_{i,n}\mu_{1}(\Xi_{0,n}^{-1}\Xi_{i,n})\le 1/\epsilon.
    \end{align*}
\end{assumption}
Under this assumption, we have $\sup_{n \in \N}\max_{i=1,\ldots,n}(\|\Xi_{0,n}^{-1}\Xi_{i,n} - I\|\vee\|\Xi_{i,n}^{-1}\Xi_{0,n} - I\|)\leq (1-\epsilon)/\epsilon$.
In summary, this assumption aims to preserve the spectral structure of $x_{t}$ even under temporal dependence.
A similar assumption has been observed in high-dimensional time-series analysis. For example, \citet{bi2022spiked} suppose a time-uniform spectral gap condition on lagged autocovariance matrices to conserve the spectral structure.

Let us introduce some notation used within this section: $\Sigma_{0:k}^{\dagger}$ has the representation $\Sigma_{0:k}^{\dagger}=\sum_{i=1}^{k}\lambda_{i}^{-1}e_{i}e_{i}^{\top}$.
In addition, we use the notation
$A_{k}=\sum_{j>k}\lambda_{j}\tilde{z}_{j}\tilde{z}_{j}$, where $\tilde{z}_{j}=\Xi_{0,n}^{-1/2}\Xi_{j,n}^{1/2}z_{j}$ for all $k=0,\ldots,p$ under Assumption \ref{asmp:coherence}. 
In Section \ref{sec:example}, we show that several stochastic processes have covariance matrices that satisfy this assumption.

\subsection{Upper bound on the excess risk}
We first provide the upper bound on the risk.
Here, we utilize the effective rank of $\Sigma$ and weighted norm of $\beta^*$, which is projected onto a subspace spanned by $\Sigma_{k^*:\infty}$ and $\Sigma_{0:k}^{\dagger}$.

\begin{theorem}\label{thm:upper}
Consider the interpolation estimator $\hat{\beta}$ in \eqref{def:interpolator} for the stochastic regression problem \eqref{def:model}.
Suppose that Assumption \ref{asmp:coherence} and $\lambda_{n+1}>0$ holds.
Then, there exist constants $b=b(\epsilon)\ge 1$ and $c=c(\epsilon)\ge 1$ dependent only on $\epsilon$ such that for all $\delta \in (0,1)$ with $\delta<1-c\exp(-n/c)$: (i) if $k^{*}\ge n/c$ with $k^{\ast}:=k^{\ast}(b)$ holds, then $\Ep_{\mathcal{E}} R(\hat{\beta})\ge 1/b^{2}c^{2}\left\|\Upsilon_{n}^{-1}\Xi_{0,n}\right\|$ holds with probability at least $1-c\exp(-n/c)$; (ii) otherwise, it holds that
\begin{align*}
    R(\hat{\beta})&\le  c\left(\left\|\beta^{\ast}\right\|_{\Sigma_{k^{*}:\infty}}^{2}+\left\|\beta^{\ast}\right\|_{\Sigma_{0:k^{*}}^{\dagger}}^{2}\left(\frac{\sum_{i>k^{*}}\lambda_{i}}{n}\right)^{2}\right)+c\|\Xi_{0,n}^{-1}\Upsilon_{n}\|\left(\frac{k^{*}}{n}+\frac{n}{R_{k^{*}}\left(\Sigma\right)}\right)\log(\delta^{-1})
\end{align*}
with probability at least $1-\delta-c\exp(-n/c)$.
\end{theorem}

The upper bound is composed of two elements: the first term bounds the bias part of the risk, and the second term represents the variance.
We will describe a way to derive this bound in Section \ref{sec:derivation}.

An important implication of this theorem is that the effect of temporal dependence structures on benign overfitting is negligible under the coherence of temporal covariance matrices.
Let us compare our results with the bound for i.i.d. sub-Gaussian settings: a combination of \citet{bartlett2020benign} and \citet{tsigler2023benign} 
gives an upper bound as
\begin{align}
    R(\hat{\beta})&\le c\left(\left\|\beta^{\ast}\right\|_{\Sigma_{k^{*}:\infty}}^{2}+\left\|\beta^{\ast}\right\|_{\Sigma_{0:k^{*}}^{\dagger}}^{2}\left(\frac{\sum_{i>k^{*}}\lambda_{i}}{n}\right)^{2}\right)+c\sigma_{\varepsilon}^{2}\left(\frac{k^{*}}{n}+\frac{n}{R_{k^{*}}\left(\Sigma\right)}\right)\log(\delta^{-1}),\label{ineq:tsigler}
\end{align}
with probability at least $1-\delta-c\exp(-n/c)$, where $\sigma_{\varepsilon}>0$ is the upper bound on the sub-Gaussian norms of $\varepsilon_{i}$.
Our bound given by Theorem \ref{thm:upper} depends on $\Sigma$ in the same manner as \eqref{ineq:tsigler} under the coherence of temporal covariances such as Assumption \ref{asmp:coherence} and $\limsup_{n\to\infty}\|\Xi_{0,n}^{-1}\Upsilon_{n}\|<\infty$.
It is worth mentioning that setting $\Xi_{0,n}=I_{n}$ and $\Upsilon_{n}=\sigma_{\varepsilon}^{2}I_{n}$ coincides the two bounds. 
 
We provide additional insights by comparing Theorem \ref{thm:upper} with other high-dimensional time-series studies:  
(i) Theorem \ref{thm:upper} depends on the closeness of the covariances of \( x_{t} \) and \( \varepsilon_{t} \), i.e., \( \|\Xi_{0,n}^{-1} \Upsilon_n\| \), rather than on the dependence between \( x_{t} \) and \( \varepsilon_{t} \).  
This is an unusual property in time-series analysis; previous studies on sparsity-based high-dimensional analysis \citep[e.g.,][]{Basu2015Regularized,wu2016performance,Wong2020Lasso} rely on the long-term dependence structures of \( x_{t} \) and/or \( \varepsilon_{t} \), but the coherence of their structures is not necessarily beneficial.  
(ii) The upper bound given by Theorem \ref{thm:upper} depends on the \( \ell^{2} \)-norm of \( \beta^* \), in contrast to \citet{Basu2015Regularized}, which depends on the \( \ell^{1} \)-norm (or the number of nonzero coordinates) of \( \beta^* \).  
(iii) The variance term increases inversely with the smallest eigenvalue of \( \Xi_{0,n} \) when we assume \( \Upsilon_{n} = I_{n} \), which indicates that the covariate process should not degenerate in the time direction.  
This point can be regarded as a time-series analog of the restricted eigenvalue condition in high-dimensional analysis \citep{bickel2009simultaneous,van2008high}.  

\begin{remark}[implicit temporal decorrelation] \label{remark:implicit_decorrelation}
Theorem \ref{thm:upper} demonstrates an implicit temporal decorrelation effect induced by the interpolation estimator \eqref{def:interpolator}.
If there exists $\sigma_{\varepsilon}>0$ such that $\Xi_{0,n}^{-1}\Upsilon_{n}=\sigma_{\varepsilon}^{2}I_{n}$, then the bound in Theorem \ref{thm:upper} is identical to that in the i.i.d.\ case \eqref{ineq:tsigler}.
This observation suggests that when $\Xi_{0,n}^{-1} \Upsilon_n = \sigma_\varepsilon^2 I_n$, the interpolation estimator \textit{implicitly decorrelates} temporal dependence, even if the true values of $\Xi_{0,n}$ or $\Upsilon_{n}$ are unknown.
Theorem \ref{thm:upper} can thus be interpreted as an approximation of this implicit decorrelation.
\end{remark}

\subsection{Lower bound on the excess risk}

We establish a lower bound for the risk.
To this end, we consider a randomized version of $\beta^*$ using a prior distribution.
We adopt this setting only in this section to derive the lower bound.
We also assume that $U=I_{p}$ without loss of generality.

\begin{assumption}[prior of $\beta^*$]\label{asmp:prior}
    Let $\Bar{\beta}\in\R^{p}$ be an arbitrary nonrandom vector and $\beta^{\ast}$ is an $\R^{p}$-valued random vector such that we have
    \begin{align*}
        \left( U^{\top}\beta^{\ast}\right)^{(i)}=R_{i}\Bar{\beta}^{(i)},\ i=1,\ldots,p,
    \end{align*}
    where $\{R_{i}\}_{i=1,\ldots,p}$ is a sequence of i.i.d.\ random variables with $P(R_{i}=1)=P(R_{i}=-1)=1/2$ and independent of other random variables.
\end{assumption}

This setting is employed in \cite{tsigler2023benign} for the i.i.d. data scenario. The response variable $\{y_t\}_{t=1}^n$ is assumed to be generated by the regression model \eqref{def:model}, conditioned on $\beta^*$. Within this framework, we derive the following lower bound, which considers the expectation of the risk with respect to $\beta^*$ as part of the results.

\begin{theorem} \label{thm:lower}
The following statements hold:
\begin{enumerate}
    \item[(i)] Under the same assumptions as Theorem \ref{thm:upper}, there exist constants $b=b(\epsilon)\ge 1$ and $c=c(\epsilon)\ge 1$ dependent only on $\epsilon$ such that if $k^{*}\le n/c$ with $k^{\ast}:=k^{\ast}(b)$, then we have
\begin{align*}
    \Ep_{\mathcal{E}} [R(\hat{\beta})]\ge \frac{1}{c\left\|\Upsilon_{n}^{-1}\Xi_{0,n}\right\|}\left(\frac{k^{*}}{n}+\frac{n}{R_{k^{*}}\left(\Sigma\right)}\right),
\end{align*}
with probability at least $1-c\exp(-n/c)$.
\item[(ii)] Under Assumption \ref{asmp:prior} and the assumptions in Theorem \ref{thm:upper}, there exist constants $b=b(\epsilon)\ge 1$ and $c=c(\epsilon)\ge 1$ dependent only on $\epsilon$ such that for any $k\in\{1,\ldots,p\}$ with $r_{k}(\Sigma)\ge bn$, with probability at least $1-10e^{-n/c}$, we have
    \begin{equation*}
        \Ep_{\mathcal{E},\beta^{\ast}}[R(\hat{\beta})]\ge \frac{1}{2}\sum_{i=1}^{p}\frac{\lambda_{i}\left(\Bar{\beta}^{(i)}\right)^{2}}{\left(1+\frac{c\lambda_{i}n}{\lambda_{k+1}r_{k}(\Sigma)}\right)^{2}}.
    \end{equation*}
\end{enumerate}
\end{theorem}

These lower bounds demonstrate the tightness of each term in the upper bound in Theorem \ref{thm:upper}.
Specifically, the first lower bound (i) corresponds to the variance term in Theorem \ref{thm:upper}, while the second bound corresponds to the bias term, up to some constants.

Regarding the comparison with the i.i.d. setting, the lower bounds match those in the i.i.d. setting (Theorem 10 of \citealp{tsigler2023benign}) up to a constant factor.
More precisely, the result provides several insights: (i) the dependence among data does not affect the bias term of the risk, as shown in Theorem \ref{thm:upper}, and (ii) the dependence affects the variance term in the bounds.

\section{Benign covariance: convergence rate analysis} \label{sec:benign}

We present sufficient conditions under which the upper bound converges to zero as \( n \to \infty \).  
Recall that both \( \Sigma \) and its eigenvalues may depend on the sample size \( n \), such that \( \Sigma = \Sigma_n \) for \( n \in \mathbb{Z} \).

\subsection{Preparation}

Further, we introduce a characterization on the (spatial) covariance, $\Sigma$, developed by \citet{bartlett2020benign}.

\begin{definition}[benign covariance] \label{def:benign_covariance_2}
    The covariance matrix $\Sigma$ is \textit{benign}, if for all $\beta_{n}^{\ast}\in\R^{p_{n}}$ with $\sup_{n\in\N}\|\beta_{n}^{\ast}\|<\infty$, the following sequences $\{\tau_n: n \in \N\}$ and $\{\eta_n: n \in \N\}$ are convergent to $0$:
    \begin{align*}
        \tau_n :=\left\|\beta_{n}^{\ast}\right\|_{\Sigma_{k^{*}:\infty}}^{2}+\left\|\beta_{n}^{\ast}\right\|_{\Sigma_{0:k^{*}}^{\dagger}}^{2}\left(\frac{\sum_{i>k^{*}}\lambda_{i}}{n}\right)^{2},\mbox{~~and~~}
        \eta_n := \max\left\{ \frac{k^*}{n}, \frac{n}{R_{k^*}(\Sigma)} \right\}.
    \end{align*}
\end{definition}

The study provides specific examples of benign covariance.
Although the analysis of covariance does not have a significant direct relationship with dependent data, we included it for completeness.

\begin{theorem}\label{thm:benign}
    Let $\lambda_{i,n}=\mu_{i}(\Sigma_{n})$ for all $i=1,\ldots,p_{n}$ and $n\in\N$.
    \begin{enumerate}
        \item[(i)] If we have $\lambda_{i,n}=i^{-\gamma}\mathbf{1}\{i\le p_{n}\}$ for $\gamma\in(0,1)$ and $n\prec p_{n}\prec n^{1/(1-\gamma)}$, then $\Sigma_{n}$ is benign.
        Moreover, we have $\tau_{n}=\mathcal{O}(p_{n}^{1-\gamma}n^{-1})$ and
        \begin{equation*}
            \eta_{n}=\begin{cases}
                \Theta_{\gamma}\left((n/p_{n})^{1/\gamma-1}+n/p_{n}\right) & \text{if }\gamma\in(0,0.5),\\
                \Theta_{\gamma}\left((n/p_{n})^{1/\gamma-1}+\frac{\log (p_{n}/n)}{p_{n}/n}\right) &\text{if }\gamma=0.5,\\
                \Theta_{\gamma}\left((n/p_{n})^{1/\gamma-1}\right) &\text{if }\gamma\in(0.5,1).
            \end{cases}
        \end{equation*}
        \item[(ii) ] If we have $\lambda_{i,n}=i^{-1}\mathbf{1}\{i\le p_{n}\}$ and $n\exp(\omega(\sqrt{n}))\prec p_{n}\prec n\exp(o(n))$, then $\Sigma_{n}$ is benign.
        Moreover, we have $\tau_{n}=\mathcal{O}(\log(p_{n}/n)/n)$ and $\eta_{n}=\Theta_{\gamma}(1/\log(p_{n}/n)+n/(\log(p_{n}/n))^{2})$.
        \item[(iii)] If we have $\lambda_{i,n}=(\gamma_{i}+\epsilon_{n})\mathbf{1}\{i\le p_{n}\}$ with $\gamma_{i}=\Theta(\exp(-i/\tau))$, $p_{n}=\omega(n)$, $\epsilon_{n}p_{n}=n\exp(-o(n))$, and $\epsilon_{n}p_{n}=o(n)$, then $\Sigma_{n}$ with $\|\Sigma_{n}\|=1$ is benign.
        Moreover, we have $\tau_{n}=\mathcal{O}(\epsilon_{n}p_{n}/n)$ and $\eta_{n}=\mathcal{O}(\log(n/\epsilon_{n}p_{n}))/n+\max\{1/n,n/p_{n}\})$.
        \item[(iv)] If we have
        \begin{equation*}
            \lambda_{i,n}=\begin{cases}
                1  & \text{if }i\le s_{n},\\
                \epsilon_{n} &\text{if }s_{n}<i\le p_{n},\\
                0 & \text{otherwise,}
            \end{cases}
        \end{equation*}
        with $\epsilon_{n}>0$, $s_{n}=o(n)$, $p_{n}=\omega(n)$, and $\epsilon_{n}p_{n}=o(n)$, then $\Sigma_{n}$ is benign.
        Moreover, we have $\tau_{n}=\mathcal{O}(\epsilon_{n}+(\epsilon_{n}p_{n}/n)^{2})$ and $\eta_{n}=\mathcal{O}(s_{n}/n+n/p_{n})$.
        \item[(v)] If we have 
        \begin{equation*}
            \lambda_{i,n}=\begin{cases}
                1  & \text{if }i=1,\\
                \epsilon_{n}\frac{1+\theta^{2}-2\theta\cos(k\pi/(p_{n}+1))}{1+\theta^{2}-2\theta\cos(\pi/(p_{n}+1))}&\text{if }1<k\le p_{n},\\
                0 & \text{otherwise,}
            \end{cases}
        \end{equation*}
            with $\theta<1$ and $\epsilon_{n}>0$, $p_{n}=\omega(n)$ and $\epsilon_{n}p_{n}=o(n)$, then $\Sigma_{n}$ is benign.
            Moreover, we have $\tau_{n}=\mathcal{O}(\epsilon_{n}+(\epsilon_{n}p_{n}/n)^{2})$ and $\eta_{n}=\mathcal{O}(1/n+n/p_{n})$.
    \end{enumerate}
\end{theorem}
In case (i), $\Sigma$ is independent of $n$; otherwise, it depends on $n$. 
Cases (i) and (ii) yield infinite nonzero eigenvalues, while cases (iii) and (iv) yield $p_n$ nonzero eigenvalues, which represent finite dimensionality increasing in $n$. 
Note that all the sequences $\{\zeta_n\}_n, \{\tau_n\}_n$ and $\{\eta_n\}_n$ are simply calculated from the proof of Theorem 31 in \citet{bartlett2020benign}.

\subsection{Convergence rate}

We derive a general form of convergence rates of the risk revealed in Section \ref{sec:result}.
In preparation, we define a notation for the temporal correlation
\begin{align}
    \nu_{n} := \|\Xi_{0,n}^{-1}\Upsilon_{n}\|. \label{def:non-degenerate_correlation_hetero}
\end{align}
In the study by \citet{Basu2015Regularized}, a similar condition was required, namely, that the lower bound on the spectral density of the sums of lagged autocovariance matrices should be positive. By contrast, we employed a non-asymptotic approach to study the eigenvalues of matrices under finite time widths.

We derive a convergence rate of risk by using the notations for the benign covariance. 
\begin{proposition}[convergence rate] \label{prop:rate_homo}
Consider the setting and assumption for Theorem \ref{thm:upper}.
Suppose that $\Sigma$ is a benign covariance, as stated in Definition \ref{def:benign_covariance_2}.
Also suppose that $\|\beta^*\| = O(1)$ as $n \to \infty$.
Then, we have
\begin{align*}
    R(\hat{\beta}) = O_P \left(\tau_n + \nu_n \eta_n\right), \mbox{~as~}n \to \infty.
\end{align*}
\end{proposition}
This rate consists of two terms: the first corresponds to the bias part, and the second part corresponds to the variance.
The bias part is determined by the volume of $k^*$ and the decay speed of the eigenvalues $\lambda_i$, while the variance part is determined by $\eta_n$ and the strength of temporal correlation $\nu_n$.
Given that the rate in the i.i.d.~case by \citet{tsigler2023benign} is given as
\begin{align*}
R(\hat{\beta}) = O_P \left( \tau_n+ \eta_n \right),
\end{align*}
we can see that $\nu_n^{-1}$ reveals the effect of temporal dependence.

\section{Proof outline of upper bounds} \label{sec:derivation}

Theorems \ref{thm:upper} and \ref{thm:lower} are primarily based on the upper and lower bounds for the bias term established by \citet{tsigler2023benign}.
The proof relies on the bias–variance decomposition of the excess risk $R(\hat{\beta})$ into a bias term $T_B$ and a variance term $T_V$, after which we bound each term separately.
Notably, this decomposition plays a crucial role in the proof of Theorem \ref{thm:upper}.

\begin{lemma} \label{lem:risk_decomp}
For any $\delta \in (0,1)$, we obtain the following with probability at least $1-\delta$:
\begin{align*}
    R(\hat{\beta}) \leq 2 \beta^{*\top} T_B \beta^* + 2  C_{\delta} \trace(T_V),
\end{align*}
where $C_{\delta} = 2\log(1/\delta) + 1$. Also,  $T_B$ and $T_V$ are defined as
\begin{align*}
    &T_B := (I - X^\top (X X^\top)^{-1} X) \Sigma (I - X^\top (X X^\top)^{-1}X), \\
    &T_V := \Upsilon_{n}^{1/2}\left(XX^{\top}\right)^{-1}X\Sigma X^{\top}\left(XX^{\top}\right)^{-1}\Upsilon_{n}^{1/2}.
\end{align*}
Furthermore, we have $\Ep_{\mathcal{E}} R(\hat{\beta}) \geq \beta^{*\top} T_B \beta^* + \trace(T_V)$.
\end{lemma}

We briefly describe the proof outline via this lemma, as well as the implicit decorrelation by the interpolator.
Specifically, we consider the simplest structure of $x_{t}$ given by \eqref{def:process_homo}.
If $\Xi_{n}$ is non-degenerate, we see that
\begin{equation*}
    \hat{\beta}=X^{\top}(XX^{\top})^{-1}Y=X^{\top}\Xi_{n}^{-1/2}(\Xi_{n}^{-1/2}XX^{\top}\Xi_{n}^{-1/2})^{-1}\Xi_{n}^{-1/2}Y.
\end{equation*}
Therefore, we can regard the interpolator $\hat{\beta}$ in \eqref{def:interpolator} is equivalent to
\begin{equation*}
    \hat{\beta}=\argmin_{\beta\in\R^{p}}\|\beta\|^{2}\text{ s.t. }\sum_{i=1}^{n}(y_{t}^{\text{deco}}-\langle \beta,x_{t}^{\text{deco}}\rangle )^{2}=0,
\end{equation*}
where
\begin{equation*}
    Y^{\text{deco}}=\left[
    \begin{matrix}
        y_{1}^{\text{deco}}\\
        \vdots\\
        y_{n}^{\text{deco}}
    \end{matrix}
    \right]=\Xi_{n}^{-1/2}\left[
    \begin{matrix}
        y_{1}\\
        \vdots\\
        y_{n}
    \end{matrix}
    \right],\ X^{\text{deco}}=\left[
    \begin{matrix}
        (x_{1}^{\text{deco}})^{\top}\\
        \vdots\\
        (x_{n}^{\text{deco}})^{\top}
    \end{matrix}
    \right]=\Xi_{n}^{-1/2}\left[
    \begin{matrix}
        x_{1}^{\top}\\
        \vdots\\
        x_{n}^{\top}
    \end{matrix}
    \right]=Z\Lambda^{1/2}U^{\top}.
\end{equation*}
Hence, we can analyze $(X^{\text{deco}},Y^{\text{deco}})$ in the same way as our observations instead of considering the observable $(X,Y)$ itself.  
It is noteworthy that $X^{\text{deco}}$ is equivalent to the i.i.d.~case considered in \citet{bartlett2020benign} and \citet{tsigler2023benign}, and that the matrix $T_{B}$ depends only on the covariates and is independent of the noise.
As a result, we obtain an upper bound for $(\beta^{\ast})^{\top}T_{B}\beta^{\ast}$ that is identical to that of \citet{tsigler2023benign}.  

We also discuss the role of Assumption \ref{asmp:coherence} in the temporal covariance matrices.  
When the covariate takes the form \eqref{def:proceess}, a more delicate analysis of random matrices with non-homogeneous temporal dependence across coordinates is required.  
Assumption \ref{asmp:coherence} allows us to address such heterogeneity in a concise manner.  
We can again transform the problem in \eqref{def:interpolator} as follows:  
\begin{equation*}
    \hat{\beta}=\argmin_{\beta\in\R^{p}}\|\beta\|^{2}\text{ s.t. }\sum_{i=1}^{n}(y_{t}^{\text{deco}}-\langle \beta,x_{t}^{\text{deco}}\rangle )^{2}=0,
\end{equation*}
where $Y^{\text{deco}}=\Xi_{0,n}^{-1/2}Y$ and $X^{\text{deco}}=[\Xi_{0,n}^{-1/2}\Xi_{1,n}^{1/2}z_{1},\ldots,\Xi_{0,n}^{-1/2}\Xi_{p,n}^{1/2}z_{p}]\Lambda^{1/2}U^{\top}$.  
Under Assumption \ref{asmp:coherence}, even though each $\Xi_{0,n}^{-1/2}\Xi_{1,n}^{1/2}z_{i}$ is not necessarily isotropic, its deviation from the isotropic case can be uniformly bounded in $n$.  
Hence, we can extend the results for the i.i.d.~case to our dependent case via a more elaborate analysis.  

Furthermore, we obtain information specific to time-series data when decomposing the variance term.  
Regarding $\trace(T_{V})$, the temporal covariance matrix of the noise term after decorrelation is given by $\Xi_{n}^{-1/2}\Upsilon_{n}\Xi_{n}^{-1/2}$.
Therefore, using $\trace(AB)\le \|A\|\trace(B)$ for positive semi-definite matrices $A$ and $B$, we obtain an upper bound on $\trace(T_{V})$ that is identical to that of \citet{tsigler2023benign}, up to the scaling factor $\|\Xi_{n}^{-1}\Upsilon_{n}\|$.  
This factor plays the same role as $\nu_{n}$ in Section \ref{sec:benign}.

\section{Example} \label{sec:example}

As an example, we analyze multiple stochastic processes and risks of the stochastic regression problem with the processes. 

\subsection{Separable vector ARMA process}

We consider a separable vector ARMA process, $x_{t}$, is defined as
\begin{align}
    x_{t}= \sum_{i=1}^{\ell_1}\left(a_{i}I_p\right)x_{t-i} + w_{t}+\sum_{i=1}^{\ell_2}\left(b_{i}I_p\right)w_{t-i},\label{def:arma_homo}
\end{align}
where $\ell_1,\ell_2 \in \N$ is a number of lag, and $a_{i},b_{i}\in\R$ is a coefficient.
Let us assume that the characteristic polynomial of the auto-regressive (AR) part, $\varpi_{1}\left(z\right):=1-\sum_{i=1}^{\ell_1}a_{i}z^{i}$, for $z\in\C$ does not have roots for any $z$ with $|z|\le 1$, nor common roots with the characteristic polynomial for the moving-average (MA) part, $\varpi_{2}\left(z\right):=1+\sum_{i=1}^{\ell_2}b_{i}z^{i}$ in $\C$.
There exists a causal MA representation $x_{t}=\sum_{j=0}^{\infty}\phi_{j}w_{t-j}$ with a sequence $\{\phi_{j}\in\R\}$, which is square-summable and satisfies $\left(1+\sum_{i}b_{i}z^{i}\right)/\left(1-\sum_{i}a_{i}z^{i}\right)=\sum_{j=0}^{\infty}\phi_{j}z^{j}$ \citep[e.g., ][Theorem 3.1.1]{BD1991Time}. 
We can see that $x_{t}$ is a separable process such that $\Xi_{n}^{\left(t,t+h\right)}=\sum_{j=0}^{\infty}\phi_{j}\phi_{j+\left|h\right|}$ for all $t=1,\ldots,n$ and $h$ with $t+h\in\left\{1,\ldots,n\right\}$ and $\Sigma=Q$.
If those assumptions are not satisfied, stationarity may not hold; see Section 3.1 of \citet{BD1991Time}.

With the conditions on $\varpi_{1}(z)$ and $\varpi_{2}(z)$, Proposition 4.5.3 of \citet{BD1991Time} yield the following  inequality for the neighboring condition:
\begin{align*}
    \min_{z\in\C:|z|=1}\left|\frac{\varpi_{2}(z)}{\varpi_{1}(z)}\right|\le \inf_{n}\mu_{n}\left(\Xi_{n}\right)\le \sup_{n}\mu_{1}\left(\Xi_{n}\right)\le \max_{z\in\C:|z|=1}\left|\frac{\varpi_{2}(z)}{\varpi_{1}(z)}\right|.
\end{align*}
This property gives the following convergence of excess risk without proof.
\begin{corollary}
Consider the setting and assumptions for Theorem \ref{thm:upper} and that the covariate follows the process \eqref{def:arma_homo}.
Suppose that $\Sigma= Q$ is a benign covariance with sequences $\{\tau_n\}_{n \in \N}$ and $\{\eta_n\}_{n \in \N}$ in Definition \ref{def:benign_covariance_2}, and $\sup_{n}\mu_{1}(\Upsilon_{n})<\infty$ holds.
Then, we have $\nu_{n}\le \sup_{n}\mu_{1}(\Upsilon_{n})/\min_{z:|z|=1}(|\varpi_{2}(z)|/|\varpi_{1}(z)|)<\infty$ and thus we get
\begin{align*}
    R(\hat{\beta}) = O_{P}\Bigl(  \tau_n+ \eta_n \Bigr).
\end{align*}
\end{corollary}
This result shows that the case of ARMA has the same convergence rate as that of independent data with the same condition on $\Sigma$.

\subsection{Non-separable vector ARMA process}

We present an example of processes in \eqref{def:arma_hetero} in Example \ref{example:vector_arma}.
For the characteristic polynomials $\varpi_{1,k}\left(z\right)$ and $\varpi_{2,k}\left(z\right)$ defined in Example \ref{example:vector_arma}, we additionally assume there exists $\epsilon > 0$ such  that $\left|\varpi_{1,k}\left(z\right)\right|$ and $\left|\varpi_{2,k}\left(z\right)\right|$ are in $[\epsilon,\epsilon^{-1}]$ for all $z\in\C$ such that $\left|z\right|=1$.
Recall that we have $\Sigma=\sum_{k=1}^{p}\lambda_{k}e_{k}e_{k}^{\top}$ with $\lambda_{k}:=q_{k}(\sum_{j=0}^{\infty}\phi_{j,k}^{2})$ where $\{\phi_{j,k}\in\R\}$  satisfies $\Xi_{k,n}^{(t,t+h)}=\sum_{j=0}^{\infty}\phi_{j,k}\phi_{j+|h|,k}/\sum_{j=0}^{\infty}\phi_{j,k}^{2}$.

The following proposition verifies the equivalence of $\Sigma$ and $Q$ and the coherence condition by Assumption \ref{asmp:coherence}.

\begin{proposition}\label{prop:armah}
The following properties hold true for \eqref{def:arma_hetero}:
\begin{enumerate}
    \item[(i)] $1\le \sum_{j=0}^{\infty}\phi_{j,k}^{2}\le \epsilon^{-4}$.
    \item[(ii)] $\epsilon^{8}\le \inf_{k,n}\mu_{n}\left(\Xi_{k,n}\right)\le \sup_{k,n}\mu_{1}\left(\Xi_{k,n}\right)<\epsilon^{-4}$.
\end{enumerate}
\end{proposition}

We see that $\Sigma$ and $Q$ are equivalent by Proposition \ref{prop:armah}-(i), and 
Assumption \ref{asmp:coherence} can be satisfied by Proposition \ref{prop:armah}-(ii).
Therefore, Theorem \ref{thm:upper} immediately gives the following result without proof:

\begin{corollary}
Consider the setting in Theorem \ref{thm:upper} with covariate  $\{x_t\}_{t}$ that follows the process \eqref{def:arma_hetero}.
Suppose that $Q$ is a benign covariance with sequences $\{\zeta_n\}_{n \in \N}$ and $\{\eta_n\}_{n \in \N}$, as stated in Definition \ref{def:benign_covariance_2}.
Then, $\Sigma=\sum_{k=1}^{p}q_{k}(\sum_{j=0}^{\infty}\phi_{j,k}^{2})e_{k}e_{k}^{\top}$ is also a benign covariance and we have
\begin{align}
    R(\hat{\beta}) = O_{P}\left(\tau_{n} + \nu_{n}\eta_n \right).
\end{align}
\end{corollary}

This upper bound does not depend on dependence-based terms, such as $\Bar{\Sigma}_n$; therefore, it can be observed that the upper bound is equivalent to that in the independent data case in \citet{bartlett2020benign}.
Additionally, the derived convergence rate is identical to that in the independent case.
In other words, the same rate is achieved for the excess risk as in the independent case because the process has short memory.

\subsection{Time-varying regression model under temporal dependence}

In time-series data analysis, it is often natural to assume that model structures, such as the autocovariance functions of $x_{t}$ and/or $\varepsilon_{t}$ \citep[e.g.,][]{kuznetsov2015learning,medeiros2016L1}, and regression coefficients \citep[e.g.,][]{zhang2012inference,dangl2012predictive,karmakar2022simultaneous}, are time-varying.
Such a heterogeneous structure of observations makes theoretical analysis significantly more difficult.
However, our analysis provides a straightforward strategy to guarantee the convergence of the estimator under heterogeneity.

As one of the simplest examples, we consider the following observation model:
\begin{equation*}
    y_{t}^{\mathfrak{s},\mathfrak{n}}=\langle \beta,\mathfrak{s}_{t}x_{t}^{\text{lat}}\rangle+\mathfrak{n}_{t}\varepsilon_{t}^{\text{lat}},\ t=1,\ldots,n,
\end{equation*}
where $\mathfrak{s}=\{\mathfrak{s}_{t}\}_{t=1}^{n}$ and $\mathfrak{n}=\{\mathfrak{n}_{t}\}_{t=1}^{n}$ are unknown intensities of the signal and noise, and $x_{t}^{\text{lat}}$ is a centered Gaussian process with the structure
\begin{equation*}
    X^{\text{lat}}=\left[
    \begin{matrix}
        (x_{1}^{\text{lat}})^{\top}\\
        \vdots\\
        (x_{n}^{\text{lat}})^{\top}
    \end{matrix}
    \right]=\left[(\Xi_{1,n}^{\text{lat}})^{1/2}z_{1},\ldots,(\Xi_{p,n}^{\text{lat}})^{1/2}z_{p}\right]\Lambda^{1/2}U^{\top},
\end{equation*}
and $\varepsilon_{t}^{\text{lat}}$ is another centered Gaussian process with the temporal covariance $\Upsilon_{n}^{\text{lat}}$.
For simplicity, we assume that $\epsilon^{\text{lat}}\le \mu_{n}(\Xi_{i,n}^{\text{lat}})\le \mu_{1}(\Xi_{1,n}^{\text{lat}})\le 1/\epsilon^{\text{lat}}$ and $\epsilon^{\text{lat}}\le \mu_{n}(\Upsilon_{n}^{\text{lat}})\le \mu_{1}(\Upsilon_{n}^{\text{lat}})\le 1/\epsilon^{\text{lat}}$ for some $\epsilon^{\text{lat}}\in(0,1]$
Note that we can extend these settings to long-memory processes as long as these matrices are coherent in the sense of Assumption \ref{asmp:coherence}.

We define an extension of the excess risk \eqref{def:risk} as follows:
\begin{equation*}
    R^{\mathfrak{s}^{\ast},\mathfrak{n}^{\ast}}(\beta)=\Ep^{\ast}\left[\left(y^{\ast}-\langle \beta, x^{\ast}\rangle \right)^{2}-\left(y^{\ast}-\langle \beta^{\ast},x^{\ast}\rangle \right)^{2}\right],
\end{equation*}
where $x^{\ast}$ and $\varepsilon^{\ast}$ are independent Gaussian random variables such that that $x^{\ast}\sim N(\mathbf{0},\Sigma)$ and $\varepsilon^{\ast}\sim N(0,\sigma_{\varepsilon}^{2})$ with $\Sigma:=U\Lambda U^{\top}$ and $\sigma_{\varepsilon}^{2}>0$, $y^{\ast}=\langle \beta, \mathfrak{s}^{\ast}x^{\ast}\rangle+\mathfrak{n}^{\ast}\varepsilon^{\ast}$,  $\mathfrak{s}^{\ast},\mathfrak{n}^{\ast}>0$, and $\Ep^{\ast}$ is the expectation with respect to $(x^{\ast},y^{\ast})$.

We let the observed covariate process $x_{t}=\mathfrak{s}_{t}x_{t}^{\text{lat}}$ and then we obtain the following corollary.

\begin{corollary}
Consider the setting in Theorem \ref{thm:upper} with covariate $\{x_t=\mathfrak{s}_{t}x_{t}^{\mathrm{lat}}\}_{t}$.
Suppose that $\Sigma=U\Lambda U^{\top}$ is a benign covariance.
Then, we have
\begin{align}
    R^{\mathfrak{s}^{\ast},\mathfrak{n}^{\ast}} = O_{P}\left(\tau_{n} + \left(\max_{t=1,\ldots,n}\frac{\mathfrak{s}_{t}^{2}}{\mathfrak{n}_{t}^{2}}\right)\eta_n \right).
\end{align}
\end{corollary}

An important consequence is that the upper and lower bounds on the excess risk given by Theorems \ref{thm:upper} and \ref{thm:lower} match if the signal-to-noise (SNR) ratios (including that of unseen data $(x^{\ast},y^{\ast})$) stay in a fixed interval.
Assumption \ref{asmp:coherence} for such observations holds for any sequence of positive numbers $\{\mathfrak{s}_{t}\}_{t=1}^{n}$.
Furthermore, Theorem 10 of \citet{tsigler2023benign} yields that the upper and lower bounds given by Theorems \ref{thm:upper} and \ref{thm:lower} coincide if for fixed $\epsilon^{\text{SNR}}\in(0,1]$, $\epsilon^{\text{SNR}}\le \min_{t=1,\ldots,n}(\mathfrak{s}_{t}/\mathfrak{n}_{t})^{2}\le \max_{t=1,\ldots,n}(\mathfrak{s}_{t}/\mathfrak{n}_{t})^{2}\le 1/\epsilon^{\text{SNR}}$.

We can extend this analysis to cases such as $\beta=\beta_{t}$ and/or $x_{t}$ with coordinate-wise time-varying structures, where ``coordinate'' means $\beta_{t}^{\top}e_{i}$ and/or $x_{t}^{\top}e_{i}$ and $e_{i}$ is a column vector of $U$.

\section{Conclusion and discussion} \label{sec:conclusion}

In this study, we investigate the excess risk of an estimator for an over-parameterized linear model with dependent time-series data. We construct an estimator using an interpolator that perfectly fits the data and measure the excess prediction risk. In addition to the concept of effective ranks of autocovariance matrices used for independent data, we develop several notions to handle temporal covariance and subsequently derive upper and lower bounds on the excess risk. This result holds in the high-dimensional regime, where \( p \gg n \), regardless of sparsity. Our analysis shows that the coherence of temporal covariance matrices affects the risk bound and the convergence rates we derive. The analysis is based on the implicit decorrelation achieved by the interpolation estimator and the control of the coherence property.

A limitation of this study is that the dependent data are assumed to be Gaussian. While assuming Gaussianity is common in dependent data analysis, several methods exist to relax this assumption. However, in our study, Gaussianity plays a crucial role in reducing data correlation to handle the bias term. In the over-parameterized setting, exploring methods to avoid the Gaussianity assumption is an important direction for future work.

An important avenue for future research is the application to more complex models, such as latent variable models. For example, \citet{bunea2022interpolating} considered a factor model in an independent data setting and showed that latent noise significantly impacts the conditions for benign overfitting. Understanding how this effect changes under dependent data is an interesting direction for further study.

\appendix

\section{Additional notation}

For the matrix $X = [x_1,...,x_n]^\top \in \R^n \otimes \R^p$ by the observations, we define its Hilbert-Schmidt norm as  $\left\|X\right\|_{\text{HS}}=\sqrt{\sum_{t=1}^{n}\left\|x_{t}\right\|^{2}}$.
For $k=2,\ldots,p-1$, we use the notation $X_{0:k}$ and $X_{k:\infty}$ for the matrix comprises of the first $k$ columns and the last $p-k$ columns respectively, and $Z_{0:k}=[z_{1}\ \cdots\ z_{k}]\in\R^{n}\otimes\R^{k}$, and $Z_{k:\infty}=[z_{k+1}\ \cdots\ z_{p}]\in\R^{n}\otimes\R^{p-k}$.
We define the following random matrices
\begin{align*}
    T_{B}&:=\left(I-X^{\top}\left(XX^{\top}\right)^{-1}X\right)\Sigma\left(I-X^{\top}\left(X X^{\top}\right)^{-1}X\right),\\
    T_{V}&:=\Upsilon_{n}^{1/2}\left(XX^{\top}\right)^{-1}X\Sigma X^{\top}\left(XX^{\top}\right)^{-1}\Upsilon_{n}^{1/2}.
\end{align*}
Let $\Tilde{\Xi}_{k,n}:=\Xi_{0,n}^{-1/2}\Xi_{k,n}\Xi_{0,n}^{-1/2}$, $\Tilde{\Upsilon}_{n}:=\Xi_{0,n}^{-1/2}\Upsilon_{n}\Xi_{0,n}^{-1/2}$, and $\Tilde{z}_{k}=[\Tilde{z}_{1,k},\ldots,\Tilde{z}_{n,k}]^{\top}$ be a random vector defined as
\begin{align*}
    \Tilde{z}_{k}&:=\Xi_{0,n}^{-1/2}\Xi_{k,n}^{1/2}z_{k}=\Xi_{0,n}^{-1/2}Xe_{k}/\sqrt{\lambda_{k}}\sim N\left(\mathbf{0},\Tilde{\Xi}_{k,n}\right).
\end{align*}
We also define some random matrices for $i,k \in \N$ satisfying $\lambda_{i}>0$ and $\lambda_{k}>0$ such that
\begin{align*}
    A=\sum_{i:\lambda_{i}>0}\lambda_{i}\Tilde{z}_{i}\Tilde{z}_{i}^{\top},\
    A_{-i}=\sum_{j\neq i:\lambda_{j}>0}\lambda_{j}\Tilde{z}_{j}\Tilde{z}_{j}^{\top},\mbox{~and~}
    A_{k}=\sum_{i>k:\lambda_{i}>0}\lambda_{i}\Tilde{z}_{i}\Tilde{z}_{i}^{\top}.
\end{align*}

\section{Proofs of main theorems}\label{sec:proofs}

We display the proof of Theorem \ref{thm:upper}.
\begin{proof}[Proof of Theorem \ref{thm:upper}]
    We first introduce a bias--variance decomposition of the excess risk, and give bounds for both the bias and variance term respectively.
    
    \textbf{(Step 1: bias--variance decomposition)} 
    By Lemmas \ref{varBLLT20LS01} and \ref{varBLLT20LS02}, with probability at least $1-e^{-t}$,
    \begin{align*}
        R(\hat{\beta})\le 2\left(\beta^{\ast}\right)^{\top}T_{B}\beta^{\ast}+2\left(4t+2\right)\trace\left(T_{V}\right),
    \end{align*}
    and with probability $1$,
    \begin{equation*}
        \Ep_{\mathcal{E}}R(\hat{\beta}){=} \left(\beta^{\ast}\right)^{\top}T_{B}\beta^{\ast}+\trace\left(T_{V}\right).
    \end{equation*}

    \textbf{(Step 2: proof of the statement (i))}
    It is obvious that with probability 1,
    \begin{equation*}
        \Ep_{\mathcal{E}}R(\hat{\beta}){=} \left(\beta^{\ast}\right)^{\top}T_{B}\beta^{\ast}+\trace\left(T_{V}\right)\ge \trace\left(T_{V}\right).
    \end{equation*}
    Lemma \ref{lem:TV_lower} gives that for some $c_{1}=c_{1}(\epsilon)>0$, for any $0\le k\le n/c$ and any $b_{1}>1$, if $r_{k}(\Sigma)< b_{1}n$, then with probability at least $1-10e^{-n/c_{1}}$, 
    \begin{equation*}
        \trace\left(T_{V}\right)\ge \mu_{n}\left(\Tilde{\Upsilon}_{n}\right)\left(k+1\right)/\left(c_{1}b_{1}^{2}n\right).
    \end{equation*}
    Since $k^{\ast}(b_{1})\ge n/c$ for this $c$ means that for any $0\le k\le n/c$, $r_{k}(\Sigma)< b_{1}n$, we obtain that by taking the maximal $k$, with probability at least $1-10e^{-n/c}$,
    \begin{equation*}
        \trace\left(T_{V}\right)\ge \frac{1}{b_{1}^{2}c_{1}^{2}\left\|\Upsilon_{n}^{-1}\Xi_{0,n}\right\|}.
    \end{equation*}
    Note that we can replace the constant $c_{1}$ with larger ones to obtain this conclusion.
    We fix $b_{1}$ later on.

    \textbf{(Step 3: proof of the statement (ii))}
    Corollary \ref{cor:TB_upper} gives that for some constants $b_{2}=b_{2}(\epsilon)\ge 1$ and $c_{2}=c_{2}(\epsilon)\ge 1$ such that for any $k\in\N$ such that $k<n/c_{2}$ and $r_{k}(\Sigma)\ge b_{2}n$, with probability at least $1-c_{2}\exp(-n/c_{2})$,
    \begin{align*}
        (\beta^{\ast})^{\top}T_{B}\beta^{\ast}&\le c_{2}\left(\left\|\beta_{k:\infty}^{\ast}\right\|_{\Sigma_{k:\infty}}^{2}+\left\|\beta_{0:k}^{\ast}\right\|_{\Sigma_{0:k}^{-1}}^{2}\left(\frac{\lambda_{k+1}r_{k}(\Sigma)}{n}\right)^{2}\right).
    \end{align*}
    Lemma \ref{lem:TV_upper} yields that there are constants $b_{3}=b_{3}\left(\epsilon\right)\ge 1$ and $c_{3}=c_{3}\left(\epsilon\right)\ge 1$ such that if $0\le k\le n/c_{3}$, $r_{k}\left(\Sigma \right)\ge b_{3}n$, then with probability at least $1-7e^{-n/c_{3}}$,
    \begin{align*}
        \trace\left(T_{V}\right)\le c_{3}\left\|\Xi_{0,n}^{-1}\Upsilon_{n}\right\|\left(\frac{k}{n}+n\frac{\sum_{i>k}\lambda_{i}^{2}}{\left(\sum_{i>k}\lambda_{i}\right)^{2}}\right).
    \end{align*}

    \textbf{(Step 4: selection of constants)} It suffices to take $c=c_{1}\vee c_{2}\vee c_{3}$ and $b=b_{1}=b_{2}\vee b_{3}$.
\end{proof}

We give the proof of Theorem \ref{thm:lower}.
\begin{proof}[Proof of Theorem \ref{thm:lower}]
    We separate the proof into those of the statements (i) and (ii).

    \textbf{(i):}
    Lemmas \ref{lem:TV_lower} and \ref{varBLLT20L11} shows that 
    there is a constant $c=c\left(\epsilon\right)>0$ such that for any $0\le k\le n/c$ and any $b>1$, if $k^{\ast}=k^{\ast}(b)<\infty$, then, with probability at least $1-10e^{-n/c}$,
        \begin{align*}
            \trace\left(T_{V}\right)\ge \frac{1}{cb^{2}\left\|\Upsilon_{n}^{-1}\Xi_{0,n}\right\|}\left(\frac{k^{\ast}}{n}+\frac{b^{2}n\sum_{i>k^{\ast}}\lambda_{i}^{2}}{\left(\lambda_{k^{\ast}+1}r_{k^{\ast}}\left(\Sigma \right)\right)^{2}}\right).
        \end{align*}
        
    \textbf{(ii):} Lemma \ref{lem:TB_lower}  combined with Lemma \ref{varBLLT20LS01} gives the conclusion.
\end{proof}

\section{Bias--variance decomposition of the excess risk}

We give the bias--variance decomposition of the excess risk shown in Lemma \ref{lem:risk_decomp}.
The following lemma is adapted from Lemma S.1 of \citet{bartlett2020benign}.
\begin{lemma}\label{varBLLT20LS01}
The excess risk of the minimum norm estimator satisfies
\begin{align*}
    R(\hat{\beta})=\Ep_{x^{\ast}}\left[\left((x^{\ast})^{\top}\left(\beta^{\ast}-\hat{\beta}\right)\right)^{2}\right]\le 2\left(\beta^{\ast}\right)^{\top}T_{B}\beta^{\ast}+2\left(\Upsilon_{n}^{-1/2}\mathcal{E}\right)^{\top}T_{V}\left(\Upsilon_{n}^{-1/2}\mathcal{E}\right)
\end{align*}
and
\begin{align*}
    \Ep_{\mathcal{E}}R(\hat{\beta}){=} \left(\beta^{\ast}\right)^{\top}T_{B}\beta^{\ast}+\trace\left(T_{V}\right).
\end{align*}
\end{lemma}

\begin{proof}[Proof of Lemma \ref{varBLLT20LS01}]
Because $y^{\ast}-(x^{\ast})^{\top}\beta^{\ast}$ is zero-mean conditionally on $x^{\ast}$,
\begin{align*}
    R(\hat{\beta})&=\Ep_{x^{\ast},y^{\ast}}\left[\left(y^{\ast}-(x^{\ast})^{\top}\hat{\beta}\right)^{2}\right]-\Ep_{x^{\ast},y^{\ast}}\left[\left(y^{\ast}-(x^{\ast})^{\top}\beta^{\ast}\right)^{2}\right]\\
    &=\Ep_{x^{\ast},y^{\ast}}\left[\left(y^{\ast}-(x^{\ast})^{\top}\beta^{\ast}+(x^{\ast})^{\top}\left(\beta^{\ast}-\hat{\beta}\right)\right)^{2}\right]-\Ep_{x^{\ast},y^{\ast}}\left[\left(y^{\ast}-(x^{\ast})^{\top}\beta^{\ast}\right)^{2}\right]\\
    &=\Ep_{x^{\ast}}\left[\left((x^{\ast})^{\top}\left(\beta^{\ast}-\hat{\beta}\right)\right)^{2}\right].
\end{align*}
The equality \eqref{eq:betahat}, the definition of $\Sigma$ and $Y=X\beta^{\ast}+\mathcal{E}$ lead to
\begin{align*}
    R(\hat{\beta})&=\Ep_{x^{\ast}}\left[\left((x^{\ast})^{\top}\left(\beta^{\ast}-X\left(XX^{\top}\right)^{-1}\left(X\beta^{\ast}+\mathcal{E}\right)\right)\right)^{2}\right]\\
    &=\Ep_{x^{\ast}}\left[\left((x^{\ast})^{\top}\left(I-X\left(XX^{\top}\right)^{-1}X\right)\beta^{\ast}-(x^{\ast})^{\top}X\left(XX^{\top}\right)^{-1}\mathcal{E}\right)^{2}\right]\\
    &\le 2\Ep_{x^{\ast}}\left[\left((x^{\ast})^{\top}\left(I-X\left(XX^{\top}\right)^{-1}X\right)\beta^{\ast}\right)^{2}\right]+2\Ep_{x^{\ast}}\left[\left((x^{\ast})^{\top}X\left(XX^{\top}\right)^{-1}\mathcal{E}\right)^{2}\right]\\
    &= 2\left(\beta^{\ast}\right)^{\top}\left(I-X^{\top}\left(XX^{\top}\right)^{-1}X\right)\Sigma\left(I-X^{\top}\left(XX^{\top}\right)^{-1}X\right)\beta^{\ast}\\
    &\quad+2\mathcal{E}^{\top}\left(XX^{\top}\right)^{-1}X\Sigma X^{\top}\left(XX^{\top}\right)^{-1}\mathcal{E}\\
    &=2\left(\beta^{\ast}\right)^{\top}T_{B}\beta^{\ast}+2\mathcal{E}^{\top}\Upsilon_{n}^{-1/2}\Upsilon_{n}^{1/2}\left(XX^{\top}\right)^{-1}X\Sigma X^{\top}\left(XX^{\top}\right)^{-1}\Upsilon_{n}^{1/2}\Upsilon_{n}^{-1/2}\mathcal{E}\\
    &=2\left(\beta^{\ast}\right)^{\top}T_{B}\beta^{\ast}+2\left(\Upsilon_{n}^{-1/2}\mathcal{E}\right)^{\top}T_{V}\left(\Upsilon_{n}^{-1/2}\mathcal{E}\right).
\end{align*}
Because $\mathcal{E}$ is centered and independent of $x^{\ast}$ and {$X$},
\begin{align*}
    \Ep_{\mathcal{E}}R(\hat{\beta})&=\Ep_{x^{\ast},\mathcal{E}}\left[\left((x^{\ast})^{\top}\left(I-X\left(XX^{\top}\right)^{-1}X\right)\beta^{\ast}-(x^{\ast})^{\top}X\left(XX^{\top}\right)^{-1}\mathcal{E}\right)^{2}\right]\\
    &=\left(\beta^{\ast}\right)^{\top}\left(I-X^{\top}\left(XX^{\top}\right)^{-1}X\right)\Sigma \left(I-X^{\top}\left(X X^{\top}\right)^{\dagger}X\right)\beta^{\ast}\\
    &\quad+\trace\left(\left(XX^{\top}\right)^{-1}X\Sigma X^{\top}\left(XX^{\top}\right)^{-1}\Ep\left[\mathcal{E}\mathcal{E}^{\top}\right]\right)\\
    &= \left(\beta^{\ast}\right)^{\top}T_{B}\beta^{\ast} + \trace\left(\Upsilon_{n}^{1/2}\left(XX^{\top}\right)^{-1}X\Sigma X^{\top}\left(XX^{\top}\right)^{-1}\Upsilon_{n}^{1/2}\right)\\
    &= \left(\beta^{\ast}\right)^{\top}T_{B}\beta^{\ast} + \trace\left(T_{V}\right).
\end{align*}
Hence we obtain the equality.
\end{proof}

To bound the tail probability of the term $T_V$, we give a corollary of Lemma S.2 of \citet{bartlett2020benign}.
\begin{lemma}\label{varBLLT20LS02}
For all $M$, $n\times n$-dimensional a.s.\ positive semi-definite random matrices, with probability at least $1-e^{-t}$, we have
\begin{align*}
    \left(\Upsilon_{n}^{-1/2}\mathcal{E}\right)^{\top}M\Upsilon_{n}^{-1/2}\mathcal{E}\le \trace\left(M\right)+2\left\|M\right\|t+2\sqrt{\left\|M\right\|^{2}t^{2}+\trace\left(M^{2}\right)t}.
\end{align*}
\end{lemma}

\begin{proof}[Proof of Lemma \ref{varBLLT20LS02}]
It follows from Lemma S.2 of \citet{bartlett2020benign} and the facts such that $\mathcal{E}$ is independent of $X$ and $\Upsilon_{n}^{-1/2}\mathcal{E}\sim N\left(0,I_{n}\right)$.
\end{proof}

As \citet{bartlett2020benign}, we see that with probability at least $1-e^{-t}$,
\begin{align*}
    \left(\Upsilon_{n}^{-1/2}\mathcal{E}\right)^{\top}T_{V}\left(\Upsilon_{n}^{-1/2}\mathcal{E}\right)\le 
    \left(4t+2\right)\trace\left(T_{V}\right).
\end{align*}

By combining these results, we obtain Lemma \ref{lem:risk_decomp}.

\section{Preliminary results for random matrices}

The following lemma is a variation of Corollary 1 of \citet{bartlett2020benign}.
\begin{lemma}\label{varBLLT20C01}
There is a universal constant $a>0$ such that for any mean-zero Gaussian random variable $\Tilde{z}\in\R^{n}$ with the covariance matrix {$\Xi$}, any random subspace $\mathscr{L}$ of $\R^{n}$ of codimension $d$ that is independent of $z$, and any $u>0$, with probability at least $1-3e^{-u}$, we have
\begin{align*}
    \left\|\Tilde{z}\right\|^{2}&\le \trace\left(\Xi\right)+a\mu_{1}\left(\Xi\right)\left(u+\sqrt{nu}\right),\\
    \left\|\Pi_{\mathscr{L}}\Tilde{z}\right\|^{2}&\ge \trace\left(\Xi\right)-a\mu_{1}\left(\Xi\right)\left(d+u+\sqrt{nu}\right),
\end{align*}
where $\Pi_{\mathscr{L}}$ is the orthogonal projection on $\mathscr{L}$.
\end{lemma}

\begin{proof}[Proof of Lemma \ref{varBLLT20C01}]
We use the notation $z:=\Xi^{-1/2}\Tilde{z}$, whose distribution is $N\left(\mathbf{0},I_{n}\right)$.
Theorem 1.1 of \citet{RV2013Hanson} verifies that there is a universal constant $c>0$ such that at least probability $1-2e^{-u}$,
\begin{align*}
    \left|z^{\top}\Xi z-\trace\left(\Xi\right)\right|\le c\left(\mu_{1}\left(\Xi\right)u+\left\|\Xi\right\|_{\mathrm{HS}}\sqrt{u}\right)\le c\mu_{1}\left(\Xi\right)\left(u+\sqrt{nu}\right).
\end{align*}
We can easily obtain the first inequality of the statement, because we have
\begin{align*}
    \left\|\Tilde{z}\right\|^{2}= z^{\top}\Xi z\le \trace\left(\Xi\right)+c\mu_{1}\left(\Xi\right)\left(u+\sqrt{nu}\right).
\end{align*}
With respect to the second inequality of the statement, we have the decomposition as \citet{bartlett2020benign}:
\begin{align*}
    \left\|\Pi_{\mathscr{L}}\Tilde{z}\right\|^{2}=\left\|\Tilde{z}\right\|^{2}-\left\|\Pi_{\mathscr{L}^{\perp}}\Tilde{z}\right\|^{2}.
\end{align*}
Note that
\begin{align*}
    \left\|\Tilde{z}\right\|^{2}=z^{\top}\Xi z\ge \trace\left(\Xi\right)-c\mu_{1}\left(\Xi\right)\left(u+\sqrt{nu}\right).
\end{align*}
Using the transpose $\Pi_{\mathscr{L}^{\perp}}^{\top}$ of the projection $\Pi_{\mathscr{L}^{\perp}}$, we define $M:=\Pi_{\mathscr{L}^{\perp}}^{\top}\Pi_{\mathscr{L}^{\perp}}$, $\left\|M\right\|=1$ and $\trace\left(M\right)=\trace\left(M^{2}\right)=d$.
Then, Lemma S.2 of \citet{bartlett2020benign} leads to the bound that at least probability $1-e^{-u}$,
\begin{align*}
    \left\|\Pi_{\mathscr{L}^{\perp}}\Tilde{z}\right\|^{2}&\le z^{\top}\Xi^{1/2}M\Xi^{1/2}z\\
    &\le \trace\left(\Xi M\right)+2\left\|\Xi^{1/2}M\Xi^{1/2}\right\|u+2\sqrt{\left\|\Xi^{1/2}M\Xi^{1/2}\right\|^{2}u^{2}+\trace\left(\left(\Xi M\right)^{2}\right)u}\\
    &\le \mu_{1}\left(\Xi\right)\trace\left(M\right)+2\mu_{1}\left(\Xi\right)u+2\sqrt{\mu_{1}^{2}\left(\Xi\right)u^{2}+\mu_{1}^{2}\left(\Xi\right)\trace\left(M^{2}\right)u}\\
    &\le \mu_{1}\left(\Xi\right)d+2\mu_{1}\left(\Xi\right)u+2\sqrt{\mu_{1}^{2}\left(\Xi\right)u^{2}+ \mu_{1}^{2}\left(\Xi\right)du}\\
    &= \mu_{1}\left(\Xi\right)\left(d+2u+2\sqrt{u\left(d+u\right)}\right)\\
    &\le \mu_{1}\left(\Xi\right)\left(2d+4u\right)
\end{align*}
by some trace inequalities (e.g., immediately obtained by von Neumann trace inequalities)
\begin{align*}
    \trace\left(\Xi M\right)
    &\le\mu_{1}\left(\Xi\right)\trace\left(M\right),\\
    \trace\left(\Xi M\Xi M\right)
    &\le \mu_{1}\left(\Xi\right)\trace\left(M\Xi M\right)
    =\mu_{1}\left(\Xi\right)\trace\left(\Xi M^{2}\right)
    \le \mu_{1}^{2}\left(\Xi\right)\trace\left(M^{2}\right).
\end{align*}
Hence
\begin{align*}
    \left\|\Pi_{\mathscr{L}}\Tilde{z}\right\|^{2}&\ge \trace\left(\Xi\right)-c\mu_{1}\left(\Xi\right)\left(u+\sqrt{nu}\right)-\mu_{1}\left(\Xi\right)\left(2d+4u\right)\\
    &\ge \trace\left(\Xi\right)-\mu_{1}\left(\Xi\right)\left(c\left(u+\sqrt{nu}\right)+2d+4u\right).
\end{align*}
We obtain the second inequality.
\end{proof}

\begin{lemma}\label{lem:bound_eigen}
Under Assumption \ref{asmp:coherence}, there is a universal constant $c>0$ such that with probability at least $1-2e^{-n/c}$, it holds that
\begin{align}
    &\frac{\epsilon}{c}\sum_{i}\lambda_{i}-c\epsilon^{-2}\lambda_{1}n\le \mu_{n}\left(A\right)\le\mu_{1}\left(A\right)\le c\left(\epsilon^{-1}\sum_{i}\lambda_{i}+\epsilon^{-2}\lambda_{1}n\right).
\end{align}
\end{lemma}

\begin{proof}[Proof of Lemma \ref{lem:bound_eigen}]
For any $v\in\R^{n}$ and $k\in\N:\lambda_{k}>0$, $v^{\top}\Tilde{z}_{k}$ is clearly $c_{0}\epsilon^{-1}\left\|v\right\|^{2}$-sub-Gaussian with a universal constant $c_{0}$.
Because for all fixed unit vector $v\in\R^{n}$, $\left(v^{\top}z_{k}\right)^{2}-v^{\top}\Xi_{k}v$ is a centered $c_{1}\epsilon^{-1}$-sub-exponential random variable with a universal constant $c_{1}>0$, Proposition 2.5.2 and Lemma 2.7.6 of \citet{Vershynin2018High} yield that there exists a universal constant $c_{2}=c_{2}\left(c_{1}\right)>0$ such that for any fixed unit vector $v\in\R^{n}$, with probability at least $1-2e^{-t}$,
\begin{align*}
    \left|v^{\top}Av-v^{\top}\left(\sum_{i}\lambda_{i}\Tilde{\Xi}_{i,n}\right)v\right|
    &\le c_{2}\epsilon^{-1}\max\left(\sup_{k:\lambda_{k}>0}\left(v^{\top}\Tilde{\Xi}_{k,n}v\right)\lambda_{k}t,\sqrt{t\sum_{i}\lambda_{i}\left(v^{\top}\Tilde{\Xi}_{i,n}v\right)}\right)\\
    &\le c_{2}\epsilon^{-1}\max\left(\sup_{k:\lambda_{k}>0}\left(v^{\top}\Tilde{\Xi}_{k,n}v\right)\lambda_{1}t,\sqrt{\sup_{k:\lambda_{k}>0}\left(v^{\top}\Tilde{\Xi}_{k,n}v\right)t\sum_{i}\lambda_{i}}\right)\\
    &\le c_{2}\epsilon^{-1}\max\left(t\epsilon^{-1}\lambda_{1},\sqrt{t\epsilon^{-1}\sum_{i}\lambda_{i}}\right)\\
    &\le c_{2}\epsilon^{-2}\max\left(t\lambda_{1},\sqrt{t\sum_{i}\lambda_{i}}\right).
\end{align*}
\citep[see also as Lemma S.9 of][]{bartlett2020benign}.
Let $\mathcal{N}$ be a $1/4$-net on $\mathcal{S}^{n-1}:=\left\{v\in\R^{n}:\left\|v\right\|=1\right\}$ such that $\left|\mathcal{N}\right|\le 9^{n}$.
The union bound argument leads to that with probability at least $1-2e^{-t}$, for all $v\in\mathcal{N}$,
\begin{align*}
    \left|v^{\top}Av-v^{\top}\left(\sum_{i}\lambda_{i}\Tilde{\Xi}_{i,n}\right)v\right|&\le c_{2}\epsilon^{-2}\max\left(\left(t+n\log9\right)\lambda_{1},\sqrt{\left(t+n\log9\right)\sum_{i}\lambda_{i}}\right).
\end{align*}
By Lemma S.8 of \citet{bartlett2020benign}, there exists a universal constant $c_{3}=c_{3}\left(c_{2}\right)>0$ such that with probability at least $1-2e^{-t}$,
\begin{align*}
    \left\|A-\sum_{i}\lambda_{i}\Tilde{\Xi}_{i,n}\right\|\le c_{3}\epsilon^{-2}\left(\left(t+n\log9\right)\lambda_{1}+\sqrt{\left(t+n\log9\right)\sum_{i}\lambda_{i}}\right).
\end{align*}
Note that
\begin{align*}
    \left\|A-\sum_{i}\lambda_{i}\Xi_{i,n}\right\|=\max_{v\in\mathcal{S}^{n-1}}\left|v^{\top}\left(A-\sum_{i}\lambda_{i}\Xi_{i,n}\right)v\right|\ge \max_{v\in\mathcal{S}^{n-1}}v^{\top}Av-\epsilon^{-1}\sum_{i}\lambda_{i},
\end{align*}
and 
\begin{align*}
    \left\|A-\sum_{i}\lambda_{i}\Xi_{i,n}\right\|=\max_{v\in\mathcal{S}^{n-1}}\left|v^{\top}\left(\sum_{i}\lambda_{i}\Xi_{i,n}-A\right)v\right|\ge \epsilon\sum_{i}\lambda_{i}-\min_{v\in\mathcal{S}^{n-1}}v^{\top}Av.
\end{align*}
The remaining discussion is parallel to \citet{bartlett2020benign}.
\end{proof}

The following corollary is a version of Lemma 4 of \citet{bartlett2020benign}.
\begin{corollary}\label{varBLLT20L04}
Under Assumption \ref{asmp:coherence}, there is a constant $c=c\left(\epsilon\right)>0$ such that for any $k\ge 0$, with probability at least $1-2e^{-n/c}$, it holds that
\begin{align*}
    \frac{1}{c}\sum_{i>k}\lambda_{i}-c\lambda_{1}n\le \mu_{n}\left(A_{k}\right)\le\mu_{1}\left(A_{k}\right)\le c\left(\sum_{i>k}\lambda_{i}+\lambda_{1}n\right).
\end{align*}
\end{corollary}

We give a variation of Lemma 5 of \citet{bartlett2020benign}.
\begin{lemma}\label{varBLLT20L05}
Under Assumption \ref{asmp:coherence}, there are constants $b=b\left(\epsilon\right)\ge1$ and $c=c\left(\epsilon\right)\ge 1$ such that for any $k\ge0$, with probability at least $1-2e^{-n/c}$, it holds that
\begin{enumerate}
    \item for all $i\ge 1$,
    \begin{align*}
        \mu_{k+1}\left(A_{-i}\right)\le \mu_{k+1}\left(A\right)\le \mu_{1}\left(A_{k}\right)\le c\left(\sum_{j>k}\lambda_{j}+n\lambda_{k+1}\right);
    \end{align*}
    \item for all $1\le i\le k$,
    \begin{align*}
        \mu_{n}\left(A\right)\ge \mu_{n}\left(A_{-i}\right)\ge \mu_{n}\left(A_{k}\right)\ge \frac{1}{c}\sum_{j>k}\lambda_{j}-cn\lambda_{k+1};
    \end{align*}
    \item if $r_{k}\left(\Sigma\right)\ge bn$, then
    \begin{align*}
        \frac{1}{c}\lambda_{k+1}r_{k}\left(\Sigma\right)\le \mu_{n}\left(A_{k}\right)\le \mu_{1}\left(A_{k}\right)\le c\lambda_{k+1}r_{k}\left(\Sigma\right).
    \end{align*}
\end{enumerate}
\end{lemma}

\begin{proof}[Proof of Lemma \ref{varBLLT20L05}]
Let us recall Corollary \ref{varBLLT20L04}: there exists a constant $c_{1}=c_{1}\left(\epsilon\right)>0$ such that for all $k\ge 0$, with probability at least $1-2e^{-n/c_{1}}$,
\begin{align*}
    \frac{1}{c_{1}}\sum_{i>k}\lambda_{i}-c_{1}\lambda_{1}n\le \mu_{n}\left(A_{k}\right)\le\mu_{1}\left(A_{k}\right)\le c_{1}\left(\sum_{i>k}\lambda_{i}+\lambda_{1}n\right).
\end{align*}

Firstly, notice that the matrix $A-A_{k}$ has its rank at most $k$.
Thus, there is a linear space $\mathscr{L}$ of dimension $n-k$ such that for all $v\in\mathscr{L}$, $v^{\top}Av=v^{\top}A_{k}v\le \mu_{1}\left(A_{k}\right)\left\|v\right\|^{2}$ and therefore $\mu_{k+1}\left(A\right)\le \mu_{1}\left(A_{k}\right)$ \citep[Lemma S.10 of][]{bartlett2020benign}.

In the second place, Lemma S.11 of \citet{bartlett2020benign} yields that for all $i$ and $j$, $\mu_{j}\left(A_{-i}\right)\le \mu_{j}\left(A\right)$.
On the other hand, for all $i\le k$, $\mu_{n}\left(A_{-i}\right)\ge\mu_{n}\left(A_{k}\right)$ by $A_{k}\preceq A_{-i}$ and Lemma S.11 of \citet{bartlett2020benign} too.

Finally, if $r_{k}\left(\Sigma \right)\ge bn$ for some $b\ge 1$, 
\begin{align*}
    c_{1}\left(\sum_{j>k}\lambda_{j}+n\lambda_{k+1}\right)&=c_{1}\left(\lambda_{k+1}r_{k}\left(\Sigma \right)+n\lambda_{k+1}\right)\\
    &\le \left(c_{1}+\frac{c_{1}}{b}\right)\lambda_{k+1}r_{k}\left(\Sigma \right),\\
    \frac{1}{c_{1}}\sum_{j>k}\lambda_{j}-c_{1}n\lambda_{k+1}&=\frac{1}{c_{1}}\lambda_{k+1}r_{k}\left(\Sigma \right)-c_{1}n\lambda_{k+1}\\
    &\ge \left(\frac{1}{c_{1}}-\frac{c_{1}}{b}\right)\lambda_{k+1}r_{k}\left(\Sigma \right).
\end{align*}
Let $b>c_{1}^{2}$ and $c>\max\left\{c_{1}+1/c_{1},1/\left(1/c_{1}-c_{1}/b\right)\right\}$; then the third statement holds.
\end{proof}

\section{Upper and lower bounds on the variance term $T_{V}$} \label{sec:proof_hetero}

Using the notation above, we give another variation of Lemma 3 of \citet{bartlett2020benign} by introducing the effect of the noise covariance.
\begin{lemma}\label{varBLLT20L03He}
Let us assume $\lambda_{n+1}>0$. We have
\begin{align*}
    \trace\left(T_{V}\right)=\sum_{i:\lambda_{i}>0}\left[\lambda_{i}^{2}\Tilde{z}_{i}^{\top}\left(\sum_{j:\lambda_{j}>0}\lambda_{j}\Tilde{z}_{j}\Tilde{z}_{j}^{\top}\right)^{-1}\Tilde{\Upsilon}_{n}\left(\sum_{j:\lambda_{j}>0}\lambda_{j}\Tilde{z}_{j}\Tilde{z}_{j}^{\top}\right)^{-1}\Tilde{z}_{i}\right],
\end{align*}
where $\Tilde{z}_{i}\sim N\left(\mathbf{0},\Tilde{\Xi}_{i,n}\right)$'s are independent $\R^{n}$-valued random variables.
In addition, we have
\begin{align*}
    \trace\left(T_{V}\right)&\le\mu_{1}\left(\Tilde{\Upsilon}_{n}\right)\sum_{i:\lambda_{i}>0}\left[\lambda_{i}^{2}\Tilde{z}_{i}^{\top}\left(\sum_{j:\lambda_{j}>0}\lambda_{j}\Tilde{z}_{j}\Tilde{z}_{j}^{\top}\right)^{-2}\Tilde{z}_{i}\right],\\
    \trace\left(T_{V}\right)&\ge\mu_{n}\left(\Tilde{\Upsilon}_{n}\right)\sum_{i:\lambda_{i}>0}\left[\lambda_{i}^{2}\Tilde{z}_{i}^{\top}\left(\sum_{j:\lambda_{j}>0}\lambda_{j}\Tilde{z}_{j}\Tilde{z}_{j}^{\top}\right)^{-2}\Tilde{z}_{i}\right].
\end{align*}
Also, for all $i$ with $\lambda_{i}>0$, we have
\begin{align*}
    \lambda_{i}^{2}\Tilde{z}_{i}^{\top}\left(\sum_{j:\lambda_{j}>0}\lambda_{j}\Tilde{z}_{j}\Tilde{z}_{j}^{\top}\right)^{-2}\Tilde{z}_{i}=\frac{\lambda_{i}^{2}\Tilde{z}_{i}^{\top}A_{-i}^{-2}\Tilde{z}_{i}}{\left(1+\lambda_{i}\Tilde{z}_{i}^{\top}A_{-i}^{-1}\Tilde{z}_{i}\right)^{2}}.
\end{align*}
\end{lemma}

\begin{proof}[Proof of Lemma \ref{varBLLT20L03He}]
We have
\begin{align*}
    XX^{\top}&=\left(\sum_{i:\lambda_{i}>0}\lambda_{i}\Xi_{i,n}^{1/2}z_{i}z_{i}^{\top}\Xi_{i,n}^{1/2}\right),\text{~~and~~}
    X\Sigma X^{\top}
    =\left(\sum_{i:\lambda_{0,i}>0}\lambda_{i}^{2}\Xi_{i,n}^{1/2}z_{i}z_{i}^{\top}\Xi_{i,n}^{1/2}\right).
\end{align*}
Therefore, it holds that
\begin{align*}
    \trace\left(T_{V}\right)
    &=\trace\left(\Upsilon_{n}^{1/2}\left(XX^{\top}\right)^{-1}X\Sigma X^{\top}\left(XX^{\top}\right)^{-1}\Upsilon_{n}^{1/2}\right)\\
    &=\trace\left(\Xi_{0,n}^{-1/2}\Upsilon_{n}\Xi_{0,n}^{-1/2}\left(\Xi_{0,n}^{-1/2}XX^{\top}\Xi_{0,n}^{-1/2}\right)^{-1}\Xi_{0,n}^{-1/2}X\Sigma X^{\top}\Xi_{0,n}^{-1/2}\left(\Xi_{0,n}^{-1/2}XX^{\top}\Xi_{0,n}^{-1/2}\right)^{-1}\right)\\
    &=\sum_{i:\lambda_{i}>0}\left[\lambda_{i}^{2}\Tilde{z}_{i}^{\top}\left(\sum_{j:\lambda_{j}>0}\lambda_{j}\Tilde{z}_{j}\Tilde{z}_{j}^{\top}\right)^{-1}\Tilde{\Upsilon}_{n}\left(\sum_{j:\lambda_{j}>0}\lambda_{j}\Tilde{z}_{j}\Tilde{z}_{j}^{\top}\right)^{-1}\Tilde{z}_{i}\right].
\end{align*}
The remaining statements immediately follow by Lemma S.3 of \citet{bartlett2020benign}.
\end{proof}
\subsection{Upper bound on $T_{V}$}
We provide the following lemma for an upper bound by developing a variation of Lemma 6 of \citet{bartlett2020benign}.
\begin{lemma} \label{lem:TV_upper}
Under Assumption \ref{asmp:coherence}, there are constants $b=b\left(\epsilon\right)\ge 1$ and $c=c\left(\epsilon\right)\ge 1$ such that if $0\le k\le n/c$, $r_{k}\left(\Sigma \right)\ge bn$, and $l\le k$, then with probability at least $1-7e^{-n/c}$, we have
\begin{align*}
    \trace\left(T_{V}\right)\le c\mu_{1}\left(\Tilde{\Upsilon}_{n}\right)\left(\frac{l}{n}+n\frac{\sum_{i>l}\lambda_{i}^{2}}{\left(\sum_{i>k}\lambda_{i}\right)^{2}}\right).
\end{align*}
\end{lemma}

\begin{proof}[Proof of Lemma \ref{lem:TV_upper}]
Fix $b=b\left(\epsilon\right)$ to its value in Lemma \ref{varBLLT20L05}.
By Lemma \ref{varBLLT20L03He}, we have
\begin{align*}
    \trace\left(T_{V}\right)&\le \mu_{1}\left(\Tilde{\Upsilon}_{n}\right)\sum_{i}\lambda_{i}^{2}\Tilde{z}_{i}^{\top}A^{-2}\Tilde{z}_{i}\\
    &\le \mu_{1}\left(\Tilde{\Upsilon}_{n}\right)\left(\sum_{i=1}^{l}\frac{\lambda_{i}^{2}\Tilde{z}_{i}^{\top}A_{-i}^{-2}\Tilde{z}_{i}}{\left(1+\lambda_{i}\Tilde{z}_{i}^{\top}A_{-i}^{-1}\Tilde{z}_{i}\right)^{2}}+\sum_{i>l}\lambda_{i}^{2}\Tilde{z}_{i}^{\top}A^{-2}\Tilde{z}_{i}\right).
\end{align*}
Firstly, let us consider the sum up to $l$.
Lemma \ref{varBLLT20L05} shows that there exists a constant $c_{1}=c_{1}\left(\epsilon\right)\ge1$ such that on the event $E_{1}$ with $P\left(E_{1}\right)\ge 1-2e^{-n/c_{1}}$, if $k$ satisfies $r_{k}\left(\Sigma \right)\ge bn$, then for all $i\le k$ $\mu_{n}\left(A_{-i}\right)\ge \lambda_{k+1}r_{k}\left(\Sigma \right)/c_{1}$, and for all $i\ge 1$ $\mu_{k+1}\left(A_{-i}\right)\le c_{1}\lambda_{k+1}r_{k}\left(\Sigma \right)$.
Hence on $E_{1}$, for all $z\in\R^{n}$ and $1\le i\le l$, we have
\begin{align*}
    \Tilde{z}^{\top} A_{-i}^{-2}\Tilde{z}&\le \frac{c_{1}^{2}\left\|\Tilde{z}\right\|^{2}}{\left(\lambda_{k+1}r_{k}\left(\Sigma \right)\right)^{2}},\\
    \Tilde{z}^{\top}A_{-i}^{-1}\Tilde{z}&\ge \left(\Pi_{\mathscr{L}_{i}}\Tilde{z}\right)^{\top} A_{-i}^{-1}\Pi_{\mathscr{L}_{i}}\Tilde{z}\ge \frac{\left\|\Pi_{\mathscr{L}_{i}}\Tilde{z}\right\|^{2}}{c_{1}\lambda_{k+1}r_{k}\left(\Sigma \right)},
\end{align*}
where $\mathscr{L}_{i}$ is the span of the $n-k$ eigenvectors of $A_{-i}$ corresponding to its smallest $n-k$ eigenvalues.
Therefore, on $E_{1}$, for all $i\le l$, we have
\begin{align*}
    \frac{\lambda_{i}^{2}\Tilde{z}_{i}^{\top}A_{-i}^{-2}\Tilde{z}_{i}}{\left(1+\lambda_{i}\Tilde{z}_{i}^{\top}A_{-i}^{-1}\Tilde{z}_{i}\right)^{2}}\le \frac{\Tilde{z}_{i}^{\top}A_{-i}^{-2}\Tilde{z}_{i}}{\left(\Tilde{z}_{i}^{\top}A_{-i}^{-1}\Tilde{z}_{i}\right)^{2}}\le c_{1}^{4}\frac{\left\|\Tilde{z}_{i}\right\|^{2}}{\left\|\Pi_{\mathscr{L}_{i}}\Tilde{z}_{i}\right\|^{2}}.
\end{align*}
We apply Lemma \ref{varBLLT20C01} $l$ times together with a union bound to show that, there exist constant $c_{0}=c_{0}\left(a,\epsilon\right)>0$, $c_{2}=c_{2}\left(a,c_{0},\epsilon\right)>0$ and $c_{3}=c_{3}\left(a,c_{0},\epsilon\right)>0$ such that on the event $E_{2}$ with $P\left(E_{2}\right)\ge 1-3e^{-t}$, for all $1\le i\le l$, we have
\begin{align*}
    \left\|\Tilde{z}_{i}\right\|^{2}&\le \epsilon^{-1}n+a\left(t+\log k+\sqrt{n\left(t+\log k\right)}\right)\le c_{2}n,\\
    \left\|\Pi_{\mathscr{L}_{i}}\Tilde{z}_{i}\right\|^{2}&\ge \epsilon n-a\left(k+t+\log k+\sqrt{n\left(t+\log k\right)}\right)\ge n/c_{3},
\end{align*}
provided that $t<n/c_{0}$ and $c>c_{0}$ owing to $\trace(\Xi_{0,n}^{-1}\Xi_{i,n})\in [\epsilon n,\epsilon^{-1}n]$ and $\log k\le n/c\le n/c_{0}$.
Then, there exists a constant $c_{4}=c_{4}\left(c_{1},c_{2},c_{3}\right)>0$ on the event $E_{1}\cap E_{2}$ with $P\left(E_{1}\cap E_{2}\right)\ge 1-5e^{-n/c_{0}}$, it holds that
\begin{align*}
    \sum_{i=1}^{l}\frac{\lambda_{i}^{2}\Tilde{z}_{i}^{\top}A_{-i}^{-2}\Tilde{z}_{i}}{\left(1+\lambda_{i}\Tilde{z}_{i}^{\top}A_{-i}^{-1}\Tilde{z}_{i}\right)^{2}}\le c_{4}\frac{l}{n}.
\end{align*}

In the second place, we consider the sum $\sum_{i>l}\lambda_{i}^{2}\Tilde{z}_{i}^{\top}A^{-2}\Tilde{z}_{i}$.
Lemma \ref{varBLLT20L05} shows that on $E_{1}$, $\mu_{n}\left(A\right)\ge \lambda_{k+1}r_{k}\left(\Sigma \right)/c_{1}$, and thus we have
\begin{align*}
    \sum_{i>l}\lambda_{i}^{2}\Tilde{z}_{i}^{\top}A^{-2}\Tilde{z}_{i}\le \frac{c_{1}^{2}\sum_{i>l}\lambda_{i}^{2}\left\|\Tilde{z}_{i}\right\|^{2}}{\left(\lambda_{k+1}r_{k}\left(\Sigma \right)\right)^{2}}.
\end{align*}
Note that we have
\begin{align*}
    \sum_{i>l}\lambda_{i}^{2}\left\|\Tilde{z}_{i}\right\|^{2}=\sum_{i>l}\lambda_{i}^{2}z_{i}^{\top}\Xi_{i,n}z_{i}\le \epsilon^{-1}\sum_{i>l}\lambda_{i}^{2}\left\|z_{i}\right\|^{2}=\epsilon^{-1}\sum_{i>l}\lambda_{i}^{2}\sum_{t=1}^{n}\left(z_{i}^{(t)}\right)^{2},
\end{align*}
where $z_{i}=\Xi_{i,n}^{-1/2}\Tilde{z}_{i}$, and $z_{i}^{(t)}\sim N\left(0,1\right)$ are i.i.d.\ random variables.
Lemma 2.7.6 of \citet{Vershynin2018High} yields that there exists a constant $c_{5}=c_{5}\left(a,c_{0},\epsilon\right)>0$ on the event $E_{3}$ with $P\left(E_{3}\right)\ge 1-2e^{-t}$, it holds that
\begin{align*}
    \sum_{i>l}\lambda_{i}^{2}\left\|\Tilde{z}_{i}\right\|^{2}&\le \epsilon^{-1}\sum_{i>l}\lambda_{i}^{2}\sum_{t=1}^{n}\left(z_{i}^{(t)}\right)^{2}\\
    &\le \epsilon^{-1} \left(n\sum_{i>l}\lambda_{i}^{2}+a\max\left\{\lambda_{l+1}^{2}t,\sqrt{tn\sum_{i>l}\lambda_{i}^{4}}\right\}\right)\\
    &\le \epsilon^{-1} \left(n\sum_{i>l}\lambda_{i}^{2}+a\max\left\{\sum_{i>l}\lambda_{i}^{2}t,\sqrt{tn}\sum_{i>l}\lambda_{i}^{2}\right\}\right)\\
    &\le c_{5}n\sum_{i>l}\lambda_{i}^{2}
\end{align*}
because $t<n/c_{0}$ holds.
Then there exists a constant $c_{6}=c_{6}\left(c_{1},c_{5}\right)>0$ such that on $E_{1}\cap E_{2}\cap E_{3}$ with $P\left(E_{1}\cap E_{2}\cap E_{3}\right)\ge1-7e^{-n/c_{0}}$, it holds that
\begin{align*}
    \sum_{i>l}\lambda_{i}^{2}\Tilde{z}_{i}^{\top}A^{-2}\Tilde{z}_{i}\le c_{6}n\frac{\sum_{i>l}\lambda_{i}^{2}}{\left(\lambda_{k+1}r_{k}\left(\Sigma \right)\right)^{2}}.
\end{align*}
Choosing $c>\max\left\{c_{0},c_{4},c_{6}\right\}$ gives the lemma.
\end{proof}

\subsection{Lower bounds on $T_{V}$}

We start with the following lemma as a variation of Lemma 8 of \citet{bartlett2020benign}.
\begin{lemma}\label{varBLLT20L08}
Under Assumption \ref{asmp:coherence}, there is a constant $c=c\left(\epsilon\right)>0$ such that for any $i\ge 1$ with $\lambda_{i}>0$ and any $0\le k\le n/c$, with probability at least $1-5e^{-n/c}$, it holds that
\begin{align*}
    \frac{\lambda_{i}^{2}\Tilde{z}_{i}^{\top}A_{-i}^{-2}\Tilde{z}_{i}}{\left(1+\lambda_{i}\Tilde{z}_{i}^{\top}A_{-i}^{-1}\Tilde{z}_{i}\right)^{2}}\ge \frac{1}{cn}\left(1+\frac{\sum_{j>k}\lambda_{j}+n\lambda_{k+1}}{n\lambda_{i}}\right)^{-2}.
\end{align*}
\end{lemma}

\begin{proof}[Proof of Lemma \ref{varBLLT20L08}]
Fix $i\ge 1$ with $\lambda_{i}>0$ and $k$ with $0\le k\le n/c_{0}$, where $c_{0}=c_{0}\left(a,\epsilon\right)>0$ is a sufficiently large constant such that there exists $c_{2}=c_{2}\left(a,c_{0},\epsilon\right)$ with
\begin{align*}
    \epsilon n-a\epsilon^{-1}\left(2n/c_{0}+n/\sqrt{c_{0}}\right)\ge n/c_{2}.
\end{align*}
By Lemma \ref{varBLLT20L05}, there exists a constant $c_{1}=c_{1}\left(\epsilon\right)\ge 1$ such that on the event $E_{1}$ with $P\left(E_{1}\right)\ge1-2e^{-n/c_{1}}$, it holds that
\begin{align*}
    \mu_{k+1}\left(A_{-i}\right)\le c_{1}\left(\sum_{j>k}\lambda_{j}+n\lambda_{k+1}\right),
\end{align*}
and hence we have
\begin{align*}
    \Tilde{z}_{i}^{\top}A_{-i}^{-1}\Tilde{z}_{i}\ge \frac{\left\|\Pi_{\mathscr{L}_{i}}\Tilde{z}_{i}\right\|^{2}}{c_{1}\left(\sum_{j>k}\lambda_{j}+n\lambda_{k+1}\right)},
\end{align*}
where $\mathscr{L}_{i}$ is the span of the $n-k$ eigenvectors of $A_{-i}$ corresponding to its smallest $n-k$ eigenvalues.
Lemma \ref{varBLLT20C01} and the definitions of $c_{0}$ and $c_{2}$ give that on the event $E_{2}$ with $P\left(E_{2}\right)\ge 1-3e^{-t}$, we have
\begin{align*}
    \left\|\Pi_{\mathscr{L}_{i}}\Tilde{z}_{i}\right\|^{2}\ge \epsilon n-a\epsilon^{-1}\left(k+t+\sqrt{tn}\right)\ge \epsilon n-a\epsilon^{-1}\left(2n/c_{0}+n/\sqrt{c_{0}}\right)\ge  n/c_{2},
\end{align*}
provided that $t<n/c_{0}$.
Hence, there exists a constant $c_{3}=c_{3}\left(c_{0},c_{1},c_{2}\right)>0$ such that on the event $E_{1}\cap E_{2}$ with $P\left(E_{1}\cap E_{2}\right)\ge 1-5e^{-n/c_{3}}$, we have
\begin{align*}
    \Tilde{z}_{i}^{\top}A_{-i}^{-1}\Tilde{z}_{i}\ge \frac{n}{c_{3}\left(\sum_{j>k}\lambda_{j}+n\lambda_{k+1}\right)},
\end{align*}
and
\begin{align*}
    1+\lambda_{i}\Tilde{z}_{i}^{\top}A_{-i}^{-1}\Tilde{z}_{i}\le \left(\frac{c_{3}\left(\sum_{j>k}\lambda_{j}+n\lambda_{k+1}\right)}{\lambda_{i}n}+1\right)\lambda_{i}\Tilde{z}_{i}^{\top}A_{-i}^{-1}\Tilde{z}_{i}.
\end{align*}
Then we have
\begin{align*}
    \frac{\lambda_{i}^{2}\Tilde{z}_{i}^{\top}A_{-i}^{-2}\Tilde{z}_{i}}{\left(1+\lambda_{i}\Tilde{z}_{i}^{\top}A_{-i}^{-1}\Tilde{z}_{i}\right)^{2}}\ge \left(\frac{c_{3}\left(\sum_{j>k}\lambda_{j}+n\lambda_{k+1}\right)}{\lambda_{i}n}+1\right)^{-2}\frac{\Tilde{z}_{i}^{\top}A_{-i}^{-2}\Tilde{z}_{i}}{\left(\Tilde{z}_{i}^{\top}A_{-i}^{-1}\Tilde{z}_{i}\right)^{2}}.
\end{align*}
The Cauchy--Schwarz inequality and Lemma \ref{varBLLT20C01} yield that there exists $c_{4}=\left(a,c_{0},\epsilon\right)>0$ such that on $E_{1}$, it holds that
\begin{align*}
    \frac{\Tilde{z}_{i}^{\top}A_{-i}^{-2}\Tilde{z}_{i}}{\left(\Tilde{z}_{i}^{\top}A_{-i}^{-1}\Tilde{z}_{i}\right)^{2}}\ge \frac{\Tilde{z}_{i}^{\top}A_{-i}^{-2}\Tilde{z}_{i}}{\left\|A_{-i}^{-1}\Tilde{z}_{i}\right\|^{2}\left\|\Tilde{z}_{i}\right\|^{2}}=\frac{1}{\left\|\Tilde{z}_{i}\right\|^{2}}\ge \frac{1}{\epsilon^{-1}n+a\epsilon^{-1}\left(t+\sqrt{nt}\right)}\ge \frac{1}{c_{4}n}.
\end{align*}
Then, there exists a constant $c_{5}=c_{5}\left(c_{3},c_{4}\right)$ such that for all $i\ge 1$ with $\lambda_{i}>0$ and $0\le k\le n/c_{0}$, with probability at least $1-5e^{-n/c_{3}}$, we have
\begin{align*}
    \frac{\lambda_{i}^{2}\Tilde{z}_{i}^{\top}A_{-i}^{-2}\Tilde{z}_{i}}{\left(1+\lambda_{i}\Tilde{z}_{i}^{\top}A_{-i}^{-1}\Tilde{z}_{i}\right)^{2}}\ge\frac{1}{c_{4}n}\left(\frac{c_{3}\left(\sum_{j>k}\lambda_{j}+n\lambda_{k+1}\right)}{\lambda_{i}n}+1\right)^{-2}\ge \frac{1}{c_{5}n}\left(1+\frac{\sum_{j>k}\lambda_{j}+n\lambda_{k+1}}{n\lambda_{i}}\right)^{-2}.
\end{align*}
Choosing $c>\max\left\{c_{0},c_{3},c_{5}\right\}$ gives the lemma.
\end{proof}

Using the above result, we develop a lower bound on $ \trace\left(T_{V}\right)$, by a variation of Lemma 10 of \citet{bartlett2020benign}.
\begin{lemma} \label{lem:TV_lower}
Under Assumption \ref{asmp:coherence}, there is a constant $c=c\left(\epsilon\right)>0$ such that for any $0\le k\le n/c$ and any $b>1$ with probability at least $1-10e^{-n/c}$, we have
\begin{enumerate}
    \item if $r_{k}\left(\Sigma \right)<bn$, then $\trace\left(T_{V}\right)\ge \mu_{n}\left(\Tilde{\Upsilon}_{n}\right)\left(k+1\right)/\left(cb^{2}n\right)$;
    \item if $r_{k}\left(\Sigma \right)\ge bn$, then
    \begin{align*}
        \trace\left(T_{V}\right)\ge \frac{\mu_{n}\left(\Tilde{\Upsilon}_{n}\right)}{cb^{2}}\min_{l\le k}\left(\frac{l}{n}+\frac{b^{2}n\sum_{i>l}\lambda_{i}^{2}}{\left(\lambda_{k+1}r_{k}\left(\Sigma \right)\right)^{2}}\right).
    \end{align*}
\end{enumerate}
In particular, if all choices of $k\le n/c$ give $r_{k}\left(\Sigma \right)<bn$, then $r_{n/c}\left(\Sigma \right)<bn$ implies that, with probability at least $1-10e^{-n/c}$, $\trace\left(T_{V}\right) \gtrsim \mu_{n}\left(\Tilde{\Upsilon}_{n}\right)$.
\end{lemma}

\begin{proof}[Proof of Lemma \ref{lem:TV_lower}]
By Lemma \ref{varBLLT20L03He}, we have
\begin{align*}
    \trace\left(T_{V}\right)\ge \mu_{n}\left(\Tilde{\Upsilon}_{n}\right)\sum_{i}\frac{\lambda_{i}^{2}\Tilde{z}_{i}^{\top}A_{-i}^{-2}z_{i}}{\left(1+\lambda_{i}\Tilde{z}_{i}^{\top}A_{-i}^{-1}\Tilde{z}_{i}\right)^{2}}
\end{align*}
and then Lemma 9 of \citet{bartlett2020benign} and Lemma \ref{varBLLT20L08} yield that there exist constants  $c_{1}=c_{1}\left(\epsilon\right)>0$ and $c_{2}=c_{2}\left(\epsilon\right)>0$ such that with probability at least $1-10e^{-n/c_{1}}$, we have
\begin{align*}
    \sum_{i}\frac{\lambda_{i}^{2}\Tilde{z}_{i}^{\top}A_{-i}^{-2}\Tilde{z}_{i}}{\left(1+\lambda_{i}\Tilde{z}_{i}^{\top}A_{-i}^{-1}\Tilde{z}_{i}\right)^{2}}&\ge \frac{1}{c_{1}n}\sum_{i}\left(1+\frac{\sum_{j>k}\lambda_{j}+n\lambda_{k+1}}{\lambda_{i}n}\right)^{-2}\\
    &= \frac{1}{c_{1}n}\sum_{i}\left(1^{-1}+\left(\frac{\lambda_{i}n}{\sum_{j>k}\lambda_{j}}\right)^{-1}+\left(\frac{\lambda_{i}}{\lambda_{k+1}}\right)^{-1}\right)^{-2}\\
    &\ge \frac{1}{c_{2}n}\sum_{i}\left(\max\left\{1^{-1},\left(\frac{\lambda_{i}n}{\sum_{j>k}\lambda_{j}}\right)^{-1},\left(\frac{\lambda_{i}}{\lambda_{k+1}}\right)^{-1}\right\}\right)^{-2}\\
    &= \frac{1}{c_{2}n}\sum_{i}\min\left\{1,\frac{\lambda_{i}^{2}n^{2}}{\left(\sum_{j>k}\lambda_{j}\right)^{2}},\frac{\lambda_{i}^{2}}{\lambda_{k+1}^{2}}\right\}\\
    &\ge \frac{1}{c_{2}b^{2}n}\sum_{i}\min\left\{1,\left(\frac{bn^{2}}{r_{k}\left(\Sigma \right)}\right)^{2}\frac{\lambda_{i}^{2}}{\lambda_{k+1}^{2}},\frac{\lambda_{i}^{2}}{\lambda_{k+1}^{2}}\right\}.
\end{align*}
If $r_{k}\left(\Sigma \right)<bn$ holds, then we have
\begin{align*}
    \trace\left(T_{V}\right)\ge \frac{\mu_{n}\left(\Tilde{\Upsilon}_{n}\right)}{c_{2}b^{2}n}\sum_{i}\min\left\{1,\frac{\lambda_{i}^{2}}{\lambda_{k+1}^{2}}\right\}=\frac{(k+1)\mu_{n}\left(\Tilde{\Upsilon}_{n}\right)}{c_{2}b^{2}n}
\end{align*}
since $\lambda_{i}\ge \lambda_{k+1}$ for $i\le k+1$.
If $r_{k}\left(\Sigma \right)\ge bn$, then we have
\begin{align*}
    \trace\left(T_{V}\right)\ge \frac{\mu_{n}\left(\Tilde{\Upsilon}_{n}\right)}{c_{2}b^{2}}\sum_{i}\min\left\{\frac{1}{n},\frac{b^{2}n\lambda_{i}^{2}}{\left(\lambda_{k+1}r_{k}\left(\Sigma \right)\right)^{2}}\right\}=\frac{\mu_{n}\left(\Tilde{\Upsilon}{n}\right)}{c_{2}b^{2}}\min_{l\le k} \left(\frac{l}{n}+\frac{b^{2}n\sum_{i>l}\lambda_{i}^{2}}{\left(\lambda_{k+1}r_{k}\left(\Sigma \right)\right)^{2}}\right)
\end{align*}
since $\left\{\lambda_{i}\right\}$ is non-increasing and then the minimizer $l$ gets restricted in $1\le l\le k$.
\end{proof}

\begin{lemma}[Lemma 11 of \citealp{bartlett2020benign}]\label{varBLLT20L11}
For any $b\ge 1$ and $k^{\ast}:=\min\left\{k:r_{k}\left(\Sigma \right)\ge bn\right\}$, if $k^{\ast}<\infty$, then we have
\begin{align*}
    \min_{l\le k^{\ast}} \left(\frac{l}{n}+\frac{b^{2}n\sum_{i>l}\lambda_{i}^{2}}{\left(\lambda_{k+1}r_{k}\left(\Sigma \right)\right)^{2}}\right)= \left(\frac{k^{\ast}}{n}+\frac{b^{2}n\sum_{i>k^{\ast}}\lambda_{i}^{2}}{\left(\lambda_{k+1}r_{k}\left(\Sigma \right)\right)^{2}}\right)=\frac{k^{\ast}}{n}+\frac{b^{2}n}{R_{k^{\ast}}\left(\Sigma \right)}.
\end{align*}
\end{lemma}

\section{Upper and lower bounds on the bias term $T_{B}$}

\subsection{Upper bound on $T_{B}$}

We first introduce a decomposition of $T_{B}$.
\begin{lemma}\label{varTB23L28}
    Fix $k<n$ and assume that $A_{k}$ is positive definite.
    Then there exists an absolute constant $c>0$ such that we have
    \begin{align*}
        (\beta^{\ast})^{\top}T_{B}\beta^{\ast}/c&\le \left\|\beta_{k:\infty}^{\ast}\right\|_{\Sigma_{k:\infty}}^{2}\\
        &\quad+\frac{\mu_{1}\left(A_{k}^{-1}\right)^{2}}{\mu_{n}\left(A_{k}^{-1}\right)^{2}}\frac{\mu_{1}\left(\Sigma_{0:k}^{-1/2}X_{0:k}^{\top}\Xi_{0,n}^{-1}X_{0:k}\Sigma_{0:k}^{1/2}\right)}{\mu_{k}\left(\Sigma_{0:k}^{-1/2}X_{0:k}^{\top}\Xi_{0,n}^{-1}X_{0:k}\Sigma_{0:k}^{1/2}\right)^{2}}\left\|\Xi_{0,n}^{-1/2}X_{k:\infty}\beta_{k:\infty}^{\ast}\right\|^{2}\\
        &\quad+\frac{\left\|\beta_{0:k}^{\ast}\right\|_{\Sigma_{0:k}^{-1}}}{\mu_{n}\left(A_{k}^{-1}\right)^{2}\mu_{k}\left(\Sigma_{0:k}^{-1/2}X_{0:k}^{\top}\Xi_{0,n}^{-1}X_{0:k}\Sigma_{0:k}^{1/2}\right)}\\
        &\quad+\lambda_{k+1}\mu_{1}\left(A^{-1}\right)\left\|\Xi_{0,n}^{-1/2}X_{k:\infty}\beta_{k:\infty}^{\ast}\right\|^{2}\\
        &\quad+\lambda_{k+1}\frac{\mu_{1}\left(A_{k}^{-1}\right)}{\mu_{n}\left(A_{k}^{-1}\right)^{2}}\frac{\mu_{1}\left(\Sigma_{0:k}^{-1/2}X_{0:k}^{\top}\Xi_{0,n}^{-1}X_{0:k}\Sigma_{0:k}^{-1/2}\right)}{\mu_{k}\left(\Sigma_{0:k}^{-1/2}X_{0:k}^{\top}\Xi_{0,n}^{-1}X_{0:k}\Sigma_{0:k}^{-1/2}\right)^{2}}\left\|\Sigma_{0:k}^{-1/2}\beta_{0:k}^{\ast}\right\|^{2}.
    \end{align*}
\end{lemma}

We give formal proof of this lemma for the sake of completeness though we know that it holds immediately.
Lemma 28 of \citet{tsigler2023benign} not using i.i.d.~properties, regarding $(\Xi_{0,n}^{-1/2}X,\Xi_{0,n}^{-1/2}y)$ as the input and output, and noticing that $\hat{\beta}$ is also the solution of
\begin{equation}
    \min_{\beta\in\R^{p}}\|\beta\|^{2}\text{ s.t. }\Xi_{0,n}^{-1/2}y=\Xi_{0,n}^{-1/2}X\beta.
\end{equation}

\begin{proof}
    We first give the following result, which is essentially identical to Eq.~(16) of \citet{tsigler2023benign}: for all $y\in\R^{n}$, we have
    \begin{equation}\label{eq:TB23E16}
        \hat{\beta}(y)_{0:k}+\left(\Xi_{0,n}^{-1/2}X_{0:k}\right)^{\top}A_{k}^{-1}\left(\Xi_{0,n}^{-1/2}X_{0:k}\right)\hat{\beta}(y)_{0:k}=\left(\Xi_{0,n}^{-1/2}X_{0:k}\right)^{\top}A_{k}^{-1}\Xi_{0,n}^{-1/2}y,
    \end{equation}
    where $\hat{\beta}(y)=(\Xi_{0,n}^{-1/2}X)^{\top}((\Xi_{0,n}^{-1/2}X)(\Xi_{0,n}^{-1/2}X)^{\top})^{-1}\Xi_{0,n}^{-1/2}y$ for arbitrary $y\in\R^{n}$.
    Eq.~(16) of \citet{tsigler2023benign} shows that
    \begin{equation*}
        \hat{\beta}(y)_{0:k}+X_{0:k}^{\top}\left(\Xi_{0,n}^{1/2}A_{k}\Xi_{0,n}^{1/2}\right)^{-1}X_{0:k}\hat{\beta}(y)_{0:k}=X_{0:k}^{\top}\left(\Xi_{0,n}^{1/2}A_{k}\Xi_{0,n}^{1/2}\right)^{-1}y.
    \end{equation*}
    We immediately notice that
    \begin{align*}
        X_{0:k}^{\top}\left(\Xi_{0,n}^{1/2}A_{k}\Xi_{0,n}^{1/2}\right)^{-1}X_{0:k}&=\left(\Xi_{0,n}^{-1/2}X_{0:k}\right)^{\top}A_{k}^{-1}\left(\Xi_{0,n}^{-1/2}X_{0:k}\right),\\
        X_{0:k}^{\top}\left(\Xi_{0,n}^{1/2}A_{k}\Xi_{0,n}^{1/2}\right)^{-1}y&=\left(\Xi_{0,n}^{-1/2}X_{0:k}\right)^{\top}A_{k}^{-1}\Xi_{0,n}^{-1/2}y,
    \end{align*}
    and 
    \begin{equation*}
        \hat{\beta}(y)=\left(\Xi_{0,n}^{-1/2}X\right)^{\top}\left(\left(\Xi_{0,n}^{-1/2}X\right)\left(\Xi_{0,n}^{-1/2}X\right)^{\top}\right)^{-1}\Xi_{0,n}^{-1/2}y=X^{\top}(XX^{\top})^{-1}y.
    \end{equation*}
    Hence we obtain the desired equality \eqref{eq:TB23E16}.

    In the second place, we bound $\|\beta_{0:k}^{\ast}-\hat{\beta}(y)_{0:k}\|_{\Sigma_{0:k}}^{2}$.
    The equality \eqref{eq:TB23E16} yields that
    \begin{equation*}
        \hat{\beta}(X\beta^{\ast})_{0:k}+\left(\Xi_{0,n}^{-1/2}X_{0:k}\right)^{\top}A_{k}^{-1}\left(\Xi_{0,n}^{-1/2}X_{0:k}\right)\hat{\beta}(X\beta^{\ast})_{0:k}=\left(\Xi_{0,n}^{-1/2}X_{0:k}\right)^{\top}A_{k}^{-1}\Xi_{0,n}^{-1/2}X\beta^{\ast}.
    \end{equation*}
    For the error vector $\zeta:=\hat{\beta}(X\beta^{\ast})-\beta^{\ast}$, we obtain
    \begin{align*}
        &\zeta_{0:k}+\beta_{0:k}^{\ast}+\left(\Xi_{0,n}^{-1/2}X_{0:k}\right)^{\top}A_{k}^{-1}\left(\Xi_{0,n}^{-1/2}X_{0:k}\right)\left(\zeta_{0:k}+\beta_{0:k}^{\ast}\right)\\
        &\quad=\left(\Xi_{0,n}^{-1/2}X_{0:k}\right)^{\top}A_{k}^{-1}\Xi_{0,n}^{-1/2}\left(X_{0:k}\beta_{0:k}^{\ast}+X_{k:\infty}\beta_{k:\infty}^{\ast}\right)
    \end{align*}
    and thus we have
    \begin{equation*}
        \zeta_{0:k}+\beta_{0:k}^{\ast}+\left(\Xi_{0,n}^{-1/2}X_{0:k}\right)^{\top}A_{k}^{-1}\left(\Xi_{0,n}^{-1/2}X_{0:k}\right)\zeta_{0:k}=\left(\Xi_{0,n}^{-1/2}X_{0:k}\right)^{\top}A_{k}^{-1}\Xi_{0,n}^{-1/2}X_{k:\infty}\beta_{k:\infty}^{\ast}.
    \end{equation*}
    Taking the inner product of both sides with $\zeta_{0:k}$ and using $\|\zeta_{0:k}\|\ge 0$, we have
    \begin{equation*}
        \zeta_{0:k}^{\top}\beta_{0:k}^{\ast}+\zeta_{0:k}^{\top}\left(\Xi_{0,n}^{-1/2}X_{0:k}\right)^{\top}A_{k}^{-1}\left(\Xi_{0,n}^{-1/2}X_{0:k}\right)\zeta_{0:k}\le \zeta_{0:k}^{\top}\left(\Xi_{0,n}^{-1/2}X_{0:k}\right)^{\top}A_{k}^{-1}\Xi_{0,n}^{-1/2}X_{k:\infty}\beta_{k:\infty}^{\ast}
    \end{equation*}
    and this gives
    \begin{align*}
        &\zeta_{0:k}^{\top}\Sigma_{0:k}^{1/2}\Sigma_{0:k}^{-1/2}\beta_{0:k}^{\ast}+\zeta_{0:k}^{\top}\Sigma_{0:k}^{1/2}\left(\Xi_{0,n}^{-1/2}X_{0:k}\Sigma_{0:k}^{-1/2}\right)^{\top}A_{k}^{-1}\left(\Xi_{0,n}^{-1/2}X_{0:k}\Sigma_{0:k}^{-1/2}\right)\Sigma_{0:k}^{1/2}\zeta_{0:k}\\
        &\quad\le \zeta_{0:k}^{\top}\Sigma_{0:k}^{1/2}\left(\Xi_{0,n}^{-1/2}X_{0:k}\Sigma_{0:k}^{-1/2}\right)^{\top}A_{k}^{-1}\Xi_{0,n}^{-1/2}X_{k:\infty}\beta_{k:\infty}^{\ast}.
    \end{align*}
    Therefore, it holds that
    \begin{align*}
        &\left\|\zeta_{0:k}\right\|_{\Sigma_{0:k}}^{2}\mu_{n}\left(A_{k}^{-1}\right)\mu_{k}\left(\Sigma_{0:k}^{-1/2}X_{0:k}\Xi_{0,n}^{-1}X_{0:k}\Sigma_{0:k}^{-1/2}\right)\\
        &\le \left\|\zeta_{0:k}\right\|_{\Sigma_{0:k}}\mu_{1}\left(A_{k}^{-1}\right)\sqrt{\mu_{1}\left(\Sigma_{0:k}^{-1/2}X_{0:k}\Xi_{0,n}^{-1}X_{0:k}\Sigma_{0:k}^{-1/2}\right)}\left\|\Xi_{0,n}^{-1/2}X_{k:\infty}\beta_{k:\infty}^{\ast}\right\|\\
        &\quad+\left\|\zeta_{0:k}\right\|_{\Sigma_{0:k}}\left\|\beta_{0:k}^{\ast}\right\|_{\Sigma_{0:k}^{-1}},
    \end{align*}
    and
    \begin{align*}
        \left\|\zeta_{0:k}\right\|_{\Sigma_{0:k}}
        &\le 
        \frac{\mu_{1}\left(A_{k}^{-1}\right)\sqrt{\mu_{1}\left(\Sigma_{0:k}^{-1/2}X_{0:k}\Xi_{0,n}^{-1}X_{0:k}\Sigma_{0:k}^{-1/2}\right)}\left\|\Xi_{0,n}^{-1/2}X_{k:\infty}\beta_{k:\infty}^{\ast}\right\|}{\mu_{n}\left(A_{k}^{-1}\right)\mu_{k}\left(\Sigma_{0:k}^{-1/2}X_{0:k}\Xi_{0,n}^{-1}X_{0:k}\Sigma_{0:k}^{-1/2}\right)}\\
        &\quad+\frac{\left\|\beta_{0:k}^{\ast}\right\|_{\Sigma_{0:k}^{-1}}}{\mu_{n}\left(A_{k}^{-1}\right)\mu_{k}\left(\Sigma_{0:k}^{-1/2}X_{0:k}\Xi_{0,n}^{-1}X_{0:k}\Sigma_{0:k}^{-1/2}\right)}.
    \end{align*}

    Finally, we consider bounds for the remaining $p-k$ components.
    We again notice that
    \begin{equation*}
        \left\|\beta_{k:\infty}^{\ast}-X_{k:\infty}^{\top}(XX^{\top})^{-1}X\beta^{\ast}\right\|_{\Sigma_{k:\infty}}^{2}= \left\|\beta_{k:\infty}^{\ast}-\left(\Xi_{0,n}^{-1/2}X_{k:\infty}\right)^{\top}A^{-1}\Xi_{0,n}^{-1/2}X\beta^{\ast}\right\|_{\Sigma_{k:\infty}}^{2}
    \end{equation*}
    and thus we have
    \begin{align*}
        &\|\beta_{k:\infty}^{\ast}-(\Xi_{0,n}^{-1/2}X_{k:\infty})^{\top}A^{-1}\Xi_{0,n}^{-1/2}X\beta^{\ast}\|_{\Sigma_{k:\infty}}^{2}\\
        &\le 3\left\|\beta_{k:\infty}^{\ast}\right\|_{\Sigma_{k:\infty}}^{2}+3\left\|\left(\Xi_{0,n}^{-1/2}X_{k:\infty}\right)^{\top}A^{-1}\left(\Xi_{0,n}^{-1/2}X_{k:\infty}\right)\beta_{k:\infty}^{\ast}\right\|_{\Sigma_{k:\infty}}^{2}\\
        &\quad+3\left\|\left(\Xi_{0,n}^{-1/2}X_{k:\infty}\right)^{\top}A^{-1}\left(\Xi_{0,n}^{-1/2}X_{0:k}\right)\beta_{0:k}^{\ast}\right\|_{\Sigma_{k:\infty}}^{2}.
    \end{align*}
    For the second term, we have
    \begin{align*}
        \left\|\left(\Xi_{0,n}^{-1/2}X_{k:\infty}\right)^{\top}A^{-1}\left(\Xi_{0,n}^{-1/2}X_{k:\infty}\right)\beta_{k:\infty}^{\ast}\right\|_{\Sigma_{k:\infty}}^{2}&=\left\|\Sigma_{k:\infty}^{1/2}X_{k:\infty}^{\top}\Xi_{0,n}^{-1/2}A^{-1}\left(\Xi_{0,n}^{-1/2}X_{k:\infty}\right)\beta_{k:\infty}^{\ast}\right\|^{2}\\
        &\le \lambda_{k+1}\left\|X_{k:\infty}^{\top}\Xi_{0,n}^{-1/2}A^{-1}\left(\Xi_{0,n}^{-1/2}X_{k:\infty}\right)\beta_{k:\infty}^{\ast}\right\|^{2}\\
        &= \lambda_{k+1}\left\|A_{k}^{1/2}A^{-1}\left(\Xi_{0,n}^{-1/2}X_{k:\infty}\right)\beta_{k:\infty}^{\ast}\right\|^{2}\\
        &\le \lambda_{k+1}\left\|A^{-1/2}\left(\Xi_{0,n}^{-1/2}X_{k:\infty}\right)\beta_{k:\infty}^{\ast}\right\|^{2}\\
        &\le \lambda_{k+1}\mu\left(A^{-1}\right)\left\|\Xi_{0,n}^{-1/2}X_{k:\infty}\beta_{k:\infty}^{\ast}\right\|^{2}.
    \end{align*}
    For the third term, \citet{tsigler2023benign} gives
    \begin{equation*}
        \left(\Xi_{0,n}^{1/2}A\Xi_{0,n}^{1/2}\right)^{-1}X_{0:k}=\left(\Xi_{0,n}^{1/2}A_{k}\Xi_{0,n}^{1/2}\right)^{-1}X_{0:k}\left(I_{k}+X_{0:k}^{\top}\left(\Xi_{0,n}^{1/2}A_{k}\Xi_{0,n}^{1/2}\right)^{-1}X_{0:k}\right)^{-1}
    \end{equation*}
    and thus we have
    \begin{equation*}
        A^{-1}\left(\Xi_{0,n}^{-1/2}X_{0:k}\right)=A_{k}\left(\Xi_{0,n}^{-1/2}X_{0:k}\right)\left(I_{k}+\left(\Xi_{0,n}^{-1/2}X_{0:k}\right)^{\top}A_{k}^{-1}\left(\Xi_{0,n}^{-1/2}X_{0:k}\right)\right)^{-1}.
    \end{equation*}
    Therefore, we have
    \begin{align*}
        &\left\|\left(\Xi_{0,n}^{-1/2}X_{k:\infty}\right)^{\top}A^{-1}\left(\Xi_{0,n}^{-1/2}X_{0:k}\right)\beta_{0:k}^{\ast}\right\|_{\Sigma_{k:\infty}}^{2}\\
        &=\left\|\left(\Xi_{0,n}^{-1/2}X_{k:\infty}\right)^{\top}A_{k}^{-1}\left(\Xi_{0,n}^{-1/2}X_{0:k}\right)\left(I_{k}+\left(\Xi_{0,n}^{-1/2}X_{0:k}\right)^{\top}A_{k}^{-1}\left(\Xi_{0,n}^{-1/2}X_{0:k}\right)\right)^{-1}\beta_{0:k}^{\ast}\right\|_{\Sigma_{k:\infty}}^{2}\\
        &=\left\|\left(\Xi_{0,n}^{-1/2}X_{k:\infty}\Sigma_{k:\infty}^{1/2}\right)^{\top}A_{k}^{-1}\left(\Xi_{0,n}^{-1/2}X_{0:k}\Sigma_{0:k}^{-1/2}\right)\right.\\
        &\qquad\left.\times\left(\Sigma_{0:k}^{-1}+\left(\Xi_{0,n}^{-1/2}X_{0:k}\Sigma_{0:k}^{-1/2}\right)^{\top}A_{k}^{-1}\left(\Xi_{0,n}^{-1/2}X_{0:k}\Sigma_{0:k}^{-1/2}\right)\right)^{-1}\Sigma_{0:k}^{-1/2}\beta_{0:k}^{\ast}\right\|^{2}\\
        &\le \frac{\left\|A_{k}^{-1/2}\left(\Xi_{0,n}^{-1/2}X_{k:\infty}\right)\Sigma_{k:\infty}\left(\Xi_{0,n}^{-1/2}X_{k:\infty}\right)^{\top}A_{k}^{-1/2}\right\|\mu_{1}\left(A_{k}^{-1/2}\right)^{2}\mu_{1}\left(\Sigma_{0:k}^{-1/2}X_{0:k}^{\top}\Xi_{0,n}^{-1}X_{0:k}\Sigma_{0:k}^{-1/2}\right)}{\mu_{k}\left(\Sigma_{0:k}^{-1}+\left(\Xi_{0,n}^{-1/2}X_{0:k}\Sigma_{0:k}^{-1/2}\right)^{\top}A_{k}^{-1}\left(\Xi_{0,n}^{-1/2}X_{0:k}\Sigma_{0:k}^{-1/2}\right)\right)^{2}}\\
        &\qquad\times \left\|\Sigma_{0:k}^{-1/2}\beta_{0:k}^{\ast}\right\|^{2}\\
        &\le \frac{\lambda_{1}\left\|A_{k}^{-1/2}\left(\Xi_{0,n}^{-1/2}X_{k:\infty}\right)\left(\Xi_{0,n}^{-1/2}X_{k:\infty}\right)^{\top}A_{k}^{-1/2}\right\|\mu_{1}\left(A_{k}^{-1}\right)\mu_{1}\left(\Sigma_{0:k}^{-1/2}X_{0:k}^{\top}\Xi_{0,n}^{-1}X_{0:k}\Sigma_{0:k}^{-1/2}\right)}{\mu_{n}\left(A_{k}^{-1}\right)^{2}\mu_{k}\left(\Sigma_{0:k}^{-1/2}X_{0:k}^{\top}\Xi_{0,n}^{-1}X_{0:k}\Sigma_{0:k}^{-1/2}\right)^{2}}\\
        &\qquad\times \left\|\Sigma_{0:k}^{-1/2}\beta_{0:k}^{\ast}\right\|^{2}\\
        &\le \frac{\lambda_{1}\mu_{1}\left(A_{k}^{-1}\right)\mu_{1}\left(\Sigma_{0:k}^{-1/2}X_{0:k}^{\top}\Xi_{0,n}^{-1}X_{0:k}\Sigma_{0:k}^{-1/2}\right)}{\mu_{n}\left(A_{k}^{-1}\right)^{2}\mu_{k}\left(\Sigma_{0:k}^{-1/2}X_{0:k}^{\top}\Xi_{0,n}^{-1}X_{0:k}\Sigma_{0:k}^{-1/2}\right)^{2}}\left\|\Sigma_{0:k}^{-1/2}\beta_{0:k}^{\ast}\right\|^{2}.
    \end{align*}
    Hence, we obtain the desired conclusion.
\end{proof}

We present an extension of Theorem 5 of \citet{tsigler2023benign}.
\begin{theorem}
    Suppose that Assumption 1 holds.
    There exists a constant $c$ dependent only on $\epsilon$ such that for any $k<n/c$, with probability at least $1-ce^{-n/c}$, if the matrix $A_{k}$ is positive definite, then we have
    \begin{align*}
        &(\beta^{\ast})^{\top}T_{B}\beta^{\ast}/c\\
        &\quad\le \left\|\beta_{k:\infty}^{\ast}\right\|_{\Sigma_{k:\infty}}^{2}\left(1+\frac{\mu_{1}(A_{k}^{-1})^{2}}{\mu_{n}(A_{k}^{-1})^{2}}+n\lambda_{k+1}\mu_{1}\left(A_{k}^{-1}\right)\right) +\left\|\beta_{0:k}^{\ast}\right\|_{\Sigma_{0:k}^{-1}}^{2}\left(\frac{1}{n^{2}\mu_{n}(A_{k}^{-1})^{2}}+\frac{\lambda_{k+1}}{n}\frac{\mu_{1}(A_{k}^{-1})}{\mu_{n}(A_{k}^{-1})^{2}}\right).
    \end{align*}
\end{theorem}

\begin{proof}
    The proof is an extension of that of Theorem 5 of \citet{tsigler2023benign}.
    By Lemma \ref{varTB23L28}, for some absolute constant $c>0$, we obtain
    \begin{align}
        (\beta^{\ast})^{\top}T_{B}\beta^{\ast}/c&\le \left\|\beta_{k:\infty}^{\ast}\right\|_{\Sigma_{k:\infty}}^{2}\\
        &\quad+\frac{\mu_{1}\left(A_{k}^{-1}\right)^{2}}{\mu_{n}\left(A_{k}^{-1}\right)^{2}}\frac{\mu_{1}\left(\Sigma_{0:k}^{-1/2}X_{0:k}^{\top}\Xi_{0,n}^{-1}X_{0:k}\Sigma_{0:k}^{1/2}\right)}{\mu_{k}\left(\Sigma_{0:k}^{-1/2}X_{0:k}^{\top}\Xi_{0,n}^{-1}X_{0:k}\Sigma_{0:k}^{1/2}\right)^{2}}\left\|\Xi_{0,n}^{-1/2}X_{k:\infty}\beta_{k:\infty}^{\ast}\right\|^{2}\\
        &\quad+\frac{\left\|\beta_{0:k}^{\ast}\right\|_{\Sigma_{0:k}^{-1}}}{\mu_{n}\left(A_{k}^{-1}\right)^{2}\mu_{k}\left(\Sigma_{0:k}^{-1/2}X_{0:k}^{\top}\Xi_{0,n}^{-1}X_{0:k}\Sigma_{0:k}^{1/2}\right)}\\
        &\quad+\lambda_{k+1}\mu_{1}\left(A^{-1}\right)\left\|\Xi_{0,n}^{-1/2}X_{k:\infty}\beta_{k:\infty}^{\ast}\right\|^{2}\\
        &\quad+\lambda_{k+1}\frac{\mu_{1}\left(A_{k}^{-1}\right)}{\mu_{n}\left(A_{k}^{-1}\right)^{2}}\frac{\mu_{1}\left(\Sigma_{0:k}^{-1/2}X_{0:k}^{\top}\Xi_{0,n}^{-1}X_{0:k}\Sigma_{0:k}^{-1/2}\right)}{\mu_{k}\left(\Sigma_{0:k}^{-1/2}X_{0:k}^{\top}\Xi_{0,n}^{-1}X_{0:k}\Sigma_{0:k}^{-1/2}\right)^{2}}\left\|\Sigma_{0:k}^{-1/2}\beta_{0:k}^{\ast}\right\|^{2}.
    \end{align}
    
    (Step 1) We first control $\mu_{1}(\Sigma_{0:k}^{-1/2}X_{0:k}^{\top}\Xi_{0,n}^{-1}X_{0:k}\Sigma_{0:k}^{1/2})$ and $\mu_{k}(\Sigma_{0:k}^{-1/2}X_{0:k}^{\top}\Xi_{0,n}^{-1}X_{0:k}\Sigma_{0:k}^{1/2})$.
    Let $J_{kn}=\diag\{\Xi_{0,n}^{-1/2}\Xi_{1,n}^{1/2},\cdots,\Xi_{0,n}^{-1/2}\Xi_{k,n}^{1/2}\}$.
    For $u\in\mathbb{S}^{k-1}$, by using the fact that $\vectorize(ABC)=(C^{\top}\otimes A)\vectorize(B)$ with the vectorization operator $\vectorize(A)=[A^{(1,1)}\ A^{(2,1)}\ \cdots\ A^{(d_{1},d_{2})}]^{\top}$ and the Kronecker product $\otimes$ for arbitrary matrices $A\in\R^{d_{1}}\otimes\R^{d_{2}},B\in\R^{d_{2}}\otimes\R^{d_{3}},C\in\R^{d_{3}}\otimes\R^{d_{4}}$, we obtain
    \begin{align*}
        \Xi_{0,n}^{-1/2}X_{0:k}\Sigma_{0:k}^{-1/2}u
        =I_{n}\left[\Xi_{0,n}^{-1/2}\Xi_{1,n}^{1/2}z_{1}\ \cdots\ \Xi_{0,n}^{-1/2}\Xi_{k,n}^{1/2}z_{k}\right]u=(u^{\top}\otimes I_{n})J_{kn}\vectorize(Z_{0:k}).
    \end{align*}
The proof of Theorem 6.3.2 of \citet{Vershynin2018High} yields that for some absolute constant $c>0$, for all $u\in\mathbb{S}^{k-1}$ and $t>0$, we have
\begin{align*}
    &P\left(\left|\left\|\Xi_{0,n}^{-1/2}X_{0:k}\Sigma_{0:k}^{-1/2}u\right\|-\left\|(u\otimes I_{n})J_{kn}\right\|_{F}\right|\ge t\right)\\
    &=P\left(\left|\left\|(u\otimes I_{n})J_{kn}\vectorize(Z_{0:k})\right\|-\left\|(u\otimes I_{n})J_{kn}\right\|_{F}\right|\ge t\right) \\
    &\le 2\exp\left(-\frac{t^{2}}{c\|(u\otimes I_{n})J_{kn}\|_{2}^{2}}\right).
\end{align*}
By the net argument \citep[Exercise 4.4.4 of][]{Vershynin2018High}, for some absolute constant $c>0$ (different than $c$ above), we have
\begin{equation*}
    P\left(\sup_{u\in\mathbb{S}^{k-1}}\left|\left\|\Xi_{0,n}^{-1/2}X_{0:k}\Sigma_{0:k}^{-1/2}u\right\|_{2}-\left\|(u\otimes I_{n})J_{kn}\right\|_{F}\right|\ge t\right)\le 2\exp\left(-\frac{t^{2}}{c\|(u\otimes I_{n})J_{kn}\|_{2}^{2}}+k\log 9\right)
\end{equation*}
and thus we have
\begin{equation*}
    P\left(\sup_{u\in\mathbb{S}^{k-1}}\left|\left\|\Xi_{0,n}^{-1/2}X_{0:k}\Sigma_{0:k}^{-1/2}u\right\|_{2}-\left\|(u\otimes I_{n})J_{kn}\right\|_{F}\right|\ge \sqrt{c\|(u\otimes I_{n})J_{kn}\|_{2}^{2}(t+k\log9)}\right)\le 2\exp\left(-t\right).
\end{equation*}
Also, it holds that
\begin{equation*}
    \left\|(u\otimes I_{n})J_{kn}\right\|_{F}=\trace\left(J_{kn}^{\top}(u^{\top}\otimes I_{n})^{\top}(u^{\top}\otimes I_{n})J_{kn}\right)^{1/2}\in\left[\sqrt{n\epsilon},\sqrt{n\epsilon^{-1}}\right],
\end{equation*}
since we have
\begin{align*}
    \trace\left((u^{\top}\otimes I_{n})^{\top}(u^{\top}\otimes I_{n})\right)
        =\trace\left(
        \left[\begin{matrix}
            u^{(1)}u^{(1)}I_{n} & \cdots & u^{(1)}u^{(k)}I_{n}\\
            \vdots & \ddots & \vdots \\
            u^{(k)}u^{(1)}I_{n} & \cdots & u^{(k)}u^{(k)}I_{n}
        \end{matrix}\right]\right)= n.
\end{align*}
Therefore, for all $t>0$ and $k\in\N$ with $n\epsilon>c\epsilon^{-1}(t+k\log9)$, with probability $1-2e^{-t}$, we obtain that
\begin{align*}
   \mu_{k}(\Sigma_{0:k}^{-1/2}X_{0:k}^{\top}\Xi_{0,n}^{-1}X_{0:k}\Sigma_{0:k}^{-1/2})&=\inf_{u\in\mathbb{S}^{k-1}}\left\|\Xi_{0,n}^{-1/2}X_{0:k}\Sigma_{0:k}^{-1/2}u\right\|_{2}^{2}\ge\left(\sqrt{n\epsilon}- \sqrt{c\epsilon^{-1}(t+k\log9)}\right)^{2},\\
   \mu_{1}(\Sigma_{0:k}^{-1/2}X_{0:k}^{\top}\Xi_{0,n}^{-1}X_{0:k}\Sigma_{0:k}^{-1/2})&=\sup_{u\in\mathbb{S}^{k-1}}\left\|\Xi_{0,n}^{-1/2}X_{0:k}\Sigma_{0:k}^{-1/2}u\right\|_{2}^{2}\le\left(\sqrt{n\epsilon^{-1}}+ \sqrt{c\epsilon^{-1}(t+k\log9)}\right)^{2}.
\end{align*}

(Step 2) We control $\Tilde{Z}_{k:\infty}$.
For some absolute constant $c>0$, for all $u\in\mathbb{S}^{n-1}$ and $v\in\R^{p-k}$, we have
\begin{equation*}
    P\left(\left|u^{\top}\Tilde{Z}_{k:\infty}v\right|\ge t\right)\le 2\exp\left(-\frac{t^{2}\epsilon}{c\|v\|^{2}}\right)
\end{equation*}
since it holds that
\begin{equation*}
    u^{\top}\Tilde{Z}_{k:\infty}v=\sum_{i=1}^{p-k}v^{(i)}u^{\top}\Xi_{0,n}^{-1/2}\Xi_{k+i,n}^{1/2}z_{k+1}\sim N\left(\mathbf{0},\sum_{i=1}^{p-k}(v^{(i)})^{2}u^{\top}\Xi_{0,n}^{-1/2}\Xi_{i+k,n}\Xi_{0,n}^{-1/2}u\right).
\end{equation*}
Therefore, we obtain that for some absolute constant $c>0$ (different than $c$ above)
\begin{equation*}
    P\left(\sup_{u\in\mathbb{S}^{n-1}}\left|u^{\top}\Tilde{Z}_{k:\infty}v\right|^{2}\ge t\right)\le 2\exp\left(-\frac{t\epsilon}{c\|v\|^{2}}+n\log9\right),
\end{equation*}
and as a result, we obtain
\begin{equation*}
    P\left(\sup_{u\in\mathbb{S}^{n-1}}\left|u^{\top}\Tilde{Z}_{k:\infty}v\right|^{2}\ge c\epsilon^{-1}\|v\|^{2}(t+n\log9)\right)\le 2\exp\left(-t\right).
\end{equation*}

(Step 3) The results above show that for some absolute constant $c>0$, for any $t>0$ and $k\in\N$ such that $t+k\log9<nc^{-1}\epsilon^{2}$, with probability $1-4e^{-t}$, we have
    \begin{align}
        (\beta^{\ast})^{\top}T_{B}\beta^{\ast}/c&\le \left\|\beta_{k:\infty}^{\ast}\right\|_{\Sigma_{k:\infty}}^{2}\\
        &\quad+\frac{\mu_{1}\left(A_{k}^{-1}\right)^{2}}{\mu_{n}\left(A_{k}^{-1}\right)^{2}}\frac{\left(\sqrt{n\epsilon^{-1}}+ \sqrt{c\epsilon^{-1}(t+k\log9)}\right)^{2}}{\left(\sqrt{n\epsilon}- \sqrt{c\epsilon^{-1}(t+k\log9)}\right)^{4}}c\epsilon^{-1}(t+n\log9)\left\|\beta_{k:\infty}^{\ast}\right\|_{\Sigma_{k:\infty}}^{2}\\
        &\quad+\frac{\left\|\beta_{0:k}^{\ast}\right\|_{\Sigma_{0:k}^{-1}}}{\mu_{n}\left(A_{k}^{-1}\right)^{2}\left(\sqrt{n\epsilon}- \sqrt{c\epsilon^{-1}(t+k\log9)}\right)^{2}}\\
        &\quad+\lambda_{k+1}\mu_{1}\left(A^{-1}\right)c\epsilon^{-1}(t+n\log9)\left\|\beta_{k:\infty}^{\ast}\right\|_{\Sigma_{k:\infty}}^{2}\\
        &\quad+\lambda_{k+1}\frac{\mu_{1}\left(A_{k}^{-1}\right)}{\mu_{n}\left(A_{k}^{-1}\right)^{2}}\frac{\left(\sqrt{n\epsilon^{-1}}+ \sqrt{c\epsilon^{-1}(t+k\log9)}\right)^{2}}{\left(\sqrt{n\epsilon}- \sqrt{c\epsilon^{-1}(t+k\log9)}\right)^{4}}\left\|\beta_{0:k}^{\ast}\right\|_{\Sigma_{0:k}^{\dagger}}^{2}.
    \end{align}
    By setting $c=c(\epsilon)$ sufficiently large, we obtain the conclusion.
\end{proof}

The theorem combined with Lemma 15 yields the following corollary.
\begin{corollary}\label{cor:TB_upper}
    Under Assumption 1, there are constants $b=b(\epsilon)\ge 1$ and $c=c(\epsilon)\ge 1$ such that for any $k\in\N$ such that $k<n/c$ and $r_{k}(\Sigma)\ge bn$, with probability at least $1-c\exp(-n/c)$, we have
    \begin{align*}
        (\beta^{\ast})^{\top}T_{B}\beta^{\ast}/c&\le \left\|\beta_{k:\infty}^{\ast}\right\|_{\Sigma_{k:\infty}}^{2}+\left\|\beta_{0:k}^{\ast}\right\|_{\Sigma_{0:k}^{-1}}^{2}\left(\frac{\lambda_{k+1}r_{k}(\Sigma)}{n}\right)^{2}.
    \end{align*}
\end{corollary}

\begin{proof}
    We consider the first term on the right-hand side.
    Lemma 15 yields that for some $b(\epsilon),c(\epsilon)\ge 1$, $\mu_{1}(A_{k}^{-1})^{2}/\mu_{n}(A_{k}^{-1})^{2}\le c(\epsilon)$, and it holds that
    \begin{equation*}
        \mu_{1}(A^{-1})\le \mu_{1}(A_{k}^{-1})=\mu_{n}(A_{k})^{-1}\le c(\epsilon)/(\lambda_{k}r_{k+1}(\Sigma))\le c(\epsilon)/(b(\epsilon)n),
    \end{equation*}
    where the first inequality holds with probability 1 by Lemma S.11 of \citet{bartlett2020benign}.

    For the second term, Lemma 15 again gives that $\mu_{n}(A_{k}^{-1})^{-1}=\mu_{1}(A_{k})\le c(\epsilon)\lambda_{k}r_{k+1}(\Sigma)$, and as the proof of Corollary 6 of \citet{tsigler2023benign}, we have
    \begin{equation*}
        \frac{\lambda_{k+1}}{n}\frac{\mu_{1}(A_{k}^{-1})}{\mu_{n}(A_{k}^{-1})^{2}}\le c(\epsilon)\frac{\lambda_{k+1}n}{n^{2}}\lambda_{k+1}r_{k}(\Sigma)\le c(\epsilon)\frac{\lambda_{k+1}^{2}b(\epsilon)^{-1}r_{k}(\Sigma)^{2}}{n^{2}}.
    \end{equation*}
    Hence, we obtain the conclusion.
\end{proof}

\subsection{Lower bound on $T_{B}$}

We now extend Lemma 8 of \citet{tsigler2023benign}; we replace some arguments on $\mu_{n}(A_{-i})$ by those of \citet{bartlett2020benign} for consistency.

\begin{lemma}\label{lem:TB_lower}
    For some $b=b(\epsilon)\ge 1$ and $c=c(\epsilon)\ge 1$ dependent only on $\epsilon$, for any $k\in\{1,\ldots,p\}$ with $r_{k}(\Sigma)\ge bn$, with probability at least $1-10e^{-n/c}$, we have
    \begin{equation*}
        \Ep_{\beta^{\ast}}[(\beta^{\ast})^{\top}T_{B}\beta^{\ast}]\ge \frac{1}{2}\sum_{i=1}^{p}\frac{\lambda_{i}(\Bar{\beta}^{(i)})^{2}}{\left(1+\frac{c(\epsilon)\lambda_{i}n}{\lambda_{k+1}r_{k}(\Sigma)}\right)^{2}}.
    \end{equation*}
\end{lemma}

\begin{proof}
    Under Assumption 2, we have
    \begin{align*}
        \Ep_{\beta^{\ast}}[(\beta^{\ast})^{\top}T_{B}\beta^{\ast}]&=\sum_{i=1}^{p}(U^{\top}T_{B}U)^{(i,i)}(\Bar{\beta}^{(i)})^{2}\\
        &=\sum_{i=1}^{p}\lambda_{i}(\Bar{\beta}^{(i)})^{2}\left(\left(1-\lambda_{i}\tilde{z}_{i}^{\top}A^{-1}\tilde{z}_{i}\right)^{2}+\sum_{j\neq i}\lambda_{j}^{2}(\tilde{z}_{j}^{\top}A^{-1}\tilde{z}_{j})^{2}\right)\\
        &\ge \sum_{i=1}^{p}\frac{\lambda_{i}(\Bar{\beta}^{(i)})^{2}}{(1+\lambda_{i}\tilde{z}_{i}^{\top}A_{-i}^{-1}\tilde{z}_{i})^{2}}.
    \end{align*}
    Lemma 12 yields that for some $c=c(\epsilon)\ge 1$, with probability at least $1-3e^{-n/c}$, $\|\tilde{z}\|^{2}\le c(\epsilon)n$.
    Lemma 15-(ii) and -(iii) give that for some $b=b(\epsilon)\ge1$ and $c=c(\epsilon)\ge1$, for any $k\in\{1,\ldots,p\}$ with $r_{k}(\Sigma)\ge bn$, with probability at least $1-2e^{-n/c}$, for any $i\in\{1,\ldots,k\}$, we have
    \begin{equation*}
        \mu_{n}(A_{-i})\ge \frac{1}{c}\lambda_{k+1}r_{k}(\Sigma).
    \end{equation*}
    Therefore, for each $i=1,\ldots,p$, for some $c=c(\epsilon)\ge 1$, with probability at least $1-5e^{-n/c}$, we have
    \begin{equation*}
        \frac{\lambda_{i}(\Bar{\beta}^{(i)})^{2}}{(1+\lambda_{i}\tilde{z}_{i}^{\top}A_{-i}^{-1}\tilde{z}_{i})^{2}}\ge \frac{\lambda_{i}(\Bar{\beta}^{(i)})^{2}}{(1+\lambda_{i}\mu_{1}(A_{-i}^{-1})\|\tilde{z}\|^{2})^{2}}=\frac{\lambda_{i}(\Bar{\beta}^{(i)})^{2}}{(1+\lambda_{i}\mu_{n}(A_{-i})^{-1}\|\tilde{z}\|^{2})^{2}}\ge \frac{\lambda_{i}(\Bar{\beta}^{(i)})^{2}}{\left(1+\frac{c\lambda_{i}n}{\lambda_{k+1}r_{k}(\Sigma)}\right)^{2}}.
    \end{equation*}
    Lemma 9 of \citet{bartlett2020benign} yields the desired conclusion.
\end{proof}

\section{Proof for convergence rate analysis}

\begin{proof}[Proof of Theorem 3]
    All the results on $\eta_{n}$ are immediate from Theorem S.14 of \citet{bartlett2020benign}. 
    Therefore, we examine only the rate of convergence of $\tau_{n}$.

    To give bounds on $\tau_{n}$ for (i)--(iii), the following fact is useful: since $k^{\ast}=k^{\ast}(b)=\min\{k\ge0;\sum_{i>k}\lambda_{i}\ge bn\lambda_{k+1}\}$ and $\sum_{i> k^{\ast}}\lambda_{i}\le \sum_{i\ge k^{\ast}}\lambda_{i}\le bn\lambda_{k^{\ast}}$ as long as $k^{\ast}\ge1$, we have
    \begin{align*}
        \tau_{n}&=\sum_{i=k^{\ast}+1}^{p_{n}}\lambda_{i}\left(u_{i}^{\top}\beta_{n}^{\ast}\right)^{2}+\sum_{i=1}^{k^{\ast}}\lambda_{i}^{-1}\left(u_{i}^{\top}\beta_{n}^{\ast}\right)^{2}\left(\frac{\sum_{i=k^{\ast}+1}^{p_{n}}\lambda_{i}}{n}\right)^{2}\\
        &\le \lambda_{k^{\ast}+1}\sum_{i=k^{\ast}+1}^{p_{n}}\left(u_{i}^{\top}\beta_{n}^{\ast}\right)^{2}+\sum_{i=1}^{k^{\ast}}\lambda_{i}^{-1}\left(u_{i}^{\top}\beta_{n}^{\ast}\right)^{2}\left(b\lambda_{k^{\ast}}\right)^{2}\\
        &\le \lambda_{k^{\ast}}\sum_{i=k^{\ast}+1}^{p_{n}}\left(u_{i}^{\top}\beta_{n}^{\ast}\right)^{2}+b^{2}\lambda_{k^{\ast}}\sum_{i=1}^{k^{\ast}}\frac{\lambda_{k^{\ast}}}{\lambda_{i}}\left(u_{i}^{\top}\beta_{n}^{\ast}\right)^{2}\\
        &\le b^{2}\lambda_{k^{\ast}}\left\|\beta_{n}^{\ast}\right\|^{2},
    \end{align*}
    where $u_{i}$ is the column vectors of $U$.
    The exact orders of $k^{\ast}$ for (i)--(iii) are given by \citet{bartlett2020benign} and we obtain the conclusions via them.

    As for $\tau_{n}$ in (iv) and (v), we just evaluate them directly.
    For (iv), since $k^{\ast}=s_{n}$ holds, we have 
    \begin{align*}
        \tau_{n}&=\epsilon_{n}\sum_{i=s_{n}+1}^{p_{n}}\left(u_{i}^{\top}\beta_{n}^{\ast}\right)^{2}+\sum_{i=1}^{s_{n}}\left(u_{i}^{\top}\beta_{n}^{\ast}\right)^{2}\left(\frac{\sum_{i=s_{n}+1}^{p_{n}}\epsilon_{n}}{n}\right)^{2}\le \left(\epsilon_{n}+\left(\frac{\epsilon_{n}p_{n}}{n}\right)^{2}\right)\left\|\beta_{n}^{\ast}\right\|^{2}.
    \end{align*}
    For (v), since $k^{\ast}=1$, we have
    \begin{align*}
        \tau_{n}&=\sum_{i=2}^{p_{n}}\lambda_{i,n}\left(u_{i}^{\top}\beta_{n}^{\ast}\right)^{2}+\left(u_{1}^{\top}\beta_{n}^{\ast}\right)^{2}\left(\frac{\sum_{i=2}^{p_{n}}\lambda_{i,n}}{n}\right)^{2}\le \left(\epsilon_{n}\frac{\left(1+\theta\right)^{2}}{\left(1-\theta\right)^{2}}+\left(\frac{\epsilon_{n}p_{n}}{n}\right)^{2}\frac{\left(1+\theta\right)^{4}}{\left(1-\theta\right)^{4}}\right)\left\|\beta_{n}^{\ast}\right\|^{2}.
    \end{align*}
    Hence, we obtain the desired result.
\end{proof}

\begin{proof}[Proof of Proposition 4]
In this proof, we study the terms in the bias part of the bound in Theorem 1.
By H\"older's inequality, we obtain
\begin{align*}
\left\|\beta^{\ast}\right\|_{\Sigma_{k^*:\infty}}^{2}+\left\|\beta^{\ast}\right\|_{\Sigma_{0:k^*}^{\dagger}}^{2}\left(\frac{\sum_{i>k^*}\lambda_{i}}{n}\right)^{2}&\le \left\|\beta^*\right\|^{2}\left(\lambda_{k^*+1}+\lambda_{k^*}^{-1}\left(\frac{\sum_{i>k^*}\lambda_{i}}{n}\right)^{2}\right)\\
&= \left\|\beta^*\right\|^{2}\left(\lambda_{k^*+1}+\lambda_{k^*}^{-1}\left(\frac{\sum_{i\ge k^*}\lambda_{i}-\lambda_{k^*}}{n}\right)^{2}\right)\\
&=\left\|\beta^*\right\|^{2}\left(\lambda_{k^*+1}+\lambda_{k^*}^{-1}\lambda_{k^*}^{2}\left(\frac{\sum_{i\ge k^*}\lambda_{i}-\lambda_{k^*}}{\lambda_{k^*}n}\right)^{2}\right)\\
&=\left\|\beta^*\right\|^{2}\left(\lambda_{k^*+1}+\lambda_{k^*}\left(\frac{r_{k^*-1}(\Sigma)-1}{n}\right)^{2}\right),
\end{align*}
which follows the definition of $r_k(\Sigma)$ in Definition 1.
With the loss of generality, we set $r_{k^*-1}(\Sigma) \geq 1$.
With the specified constant $b\ge 1$  with $k^{\ast}=k^{\ast}(b)\ge 1$, the relation $r_{k^* - 1}(\Sigma) \leq bn$ by Definition 1 yields
\begin{align*}
    \left\|\beta^*\right\|^{2}\left(\lambda_{k^{\ast}+1}+\lambda_{k^{\ast}}\left(\frac{r_{k^{\ast}-1}(\Sigma)-1}{n}\right)^{2}\right)
    &\le \left\|\beta^*\right\|^{2}\left(\lambda_{k^{\ast}+1}+\lambda_{k^{\ast}}\left(\frac{bn-1}{n}\right)^{2}\right)\\
    &\le \left \|\beta^*\right\|^{2}\left(\lambda_{k^{\ast}+1}+\lambda_{k^{\ast}}b^{2}\right) \\
    & \leq c \|\beta^*\|^{2} \lambda_{k^*} \\
    & \leq c \|\beta^*\|^{2} \frac{\sum_{i \geq k^*} \lambda_i}{bn} \\
    & \leq c \|\beta^*\|^{2} \frac{\trace (\Sigma)}{bn},
\end{align*}
where $c := (1 + b^2)$ is a constant.
The last inequality follows $\lambda_{k^* + 1} \leq \lambda_{k^*}$.

About the variance part of the upper bound in Theorem 1, the result is obvious.
\end{proof}

\section{Proofs for the examples}

We give proofs for lemmas and propositions for the example of processes presented in Section 6.

\begin{proof}[Proof for Proposition 7]
Let us begin with showing (i).
Because $\phi_{0,k}=1$ for all $k$ by the definition, the lower bound is obvious.
By the assumptions of the unit variance of white noises and spectral densities, Theorem 4.4.2 and Proposition 4.5.3 of \citet{BD1991Time} lead to
\begin{align*}
    \epsilon^{4}\le \inf_{z\in\C:\left|z\right|=1}\frac{\left|1+\sum_{j=1}^{\ell_{2}}\varphi_{j,k}z^{j}\right|^{2}}{\left|1-\sum_{j=1}^{\ell_{1}}\rho_{j,k}z^{j}\right|^{2}}\le \sum_{j=0}^{\infty}\phi_{j,k}^{2}\le \sup_{z\in\C:\left|z\right|=1}\frac{\left|1+\sum_{j=1}^{\ell_{2}}\varphi_{j,k}z^{j}\right|^{2}}{\left|1-\sum_{j=1}^{\ell_{1}}\rho_{j,k}z^{j}\right|^{2}}\le \epsilon^{-4}
\end{align*}
because $\sum_{j=0}^{\infty}\phi_{j,k}^{2}$ equals to all the diagonal elements of the corresponding autocovariance matrix.

(ii) follows from the same argument as (i) and the fact that $\sum_{j=0}^{\infty}\phi_{j,k}^{2}$ is the inverse variance of the corresponding noise terms in each coordinate process $\left\{\left(x_{t}^{\top}e_{k}\right)/\sqrt{\lambda_{k}}:t=1,\ldots,n\right\}$.
\end{proof}

\section*{Acknowledgements}
The authors would like to thank the anonymous referees, an Associate Editor, and the Editor for their constructive comments that improved the quality of this paper.
We thank Dr. Pierre Alquier for the fruitful discussion.

\section*{Funding}
The first author was supported by JSPS KAKENHI (21K20318).
The second author was supported by JSPS KAKENHI (21K11780), JST CREST (JPMJCR21D2), and JST FOREST (JPMJFR216I).

\bibliographystyle{apecon}
\bibliography{main}

\end{document}